\documentclass[11pt,reqno]{amsart}
\usepackage{amssymb}
\usepackage{enumerate}
\usepackage{amsmath,amssymb,amsfonts,amsthm,
	latexsym, amscd, epsfig, epic,color}
\usepackage{mathtools}
\usepackage{youngtab}
\usepackage{extarrows}
\usepackage{enumitem}
\usepackage{enumerate}
\usepackage{mathrsfs}
\usepackage{eucal}
\usepackage{eufrak}
\usepackage{xspace}
\usepackage{a4wide}
\usepackage{colordvi}
\usepackage{comment}
\usepackage[section]{placeins}
\usepackage{hyperref,eucal}
\usepackage{tabularx}
\usepackage{tikz}\usetikzlibrary{matrix,arrows}
\usepackage{ytableau} \usepackage{chemfig}
\usepackage{float}
\usepackage{youngtab}
\usepackage{mathrsfs}
\usepackage{caption}
\usepackage{epsfig}
\usepackage{stfloats}
\usepackage{booktabs}
\usepackage{changepage}
\usepackage{epstopdf}
\usepackage{graphicx}
\theoremstyle{plain}

\theoremstyle{plain}
\newtheorem{theorem}{Theorem}[section]
\newtheorem{corollary}{Corollary}[section]

\newtheorem{lemma}{Lemma}[section]
\newtheorem{proposition}{Proposition}[section]

\theoremstyle{definition}
\newtheorem{definition}{Definition}[section]

\theoremstyle{example}
\newtheorem{example}{Example}[section]

\theoremstyle{openproblem}

\newtheorem{remark}{Remark}[section]
\numberwithin{equation}{section}
\numberwithin{figure}{section}

\def\T{\CMcal{T}}

\def\A{\mathcal{A}}

\def\S{\CMcal{S}}
\def\Q{\CMcal{Q}}
\def\inv{\mathrm{inv}}
\def\coinv{\mathrm{coinv}}
\def\quinv{\mathrm{quinv}}

\def\arm{{\mathrm{arm}}}
\def\leg{{\mathrm{leg}}}
\def\maj{\mathrm{maj}}

\def\Des{\mathrm{Des}}
\def\des{\mathrm{des}}
\def\South{\mathrm{South}}

\def\dg{\mathrm{dg}}
\def\sym{\mathrm{Sym}}

\def\sort{\pi}

\makeatletter
\@namedef{subjclassname@2020}{\textup{2020} Mathematics Subject Classification}
\makeatother

\allowdisplaybreaks[4]

\numberwithin{equation}{section}
\numberwithin{figure}{section}

\title[Monomial expansions of modified Macdonald polynomials]{Monomial expansions of modified Macdonald polynomials}

\author{Emma Yu Jin}
\address{School of Mathematical Sciences, Xiamen University, Xiamen 361005, China}
\email{yjin@xmu.edu.cn}

\author{Xiaowei Lin}
\address{School of Mathematical Sciences, Xiamen University, Xiamen 361005, China}
\email{linxiaoweiqing@126.com}



\begin{document}
	\begin{abstract}
		We discover a family $\A$ of statistics on fillings of any given Young diagram and prove new combinatorial formulas for modified Macdonald polynomials, that is, 
		$$\widetilde{H}_{\lambda}(X;q,t)=\sum_{\sigma\in \T(\lambda)}x^{\sigma}t^{\maj(\sigma)}q^{\eta(\sigma)}$$
		for each statistic $\eta\in \A$. Building upon this new formula, we establish four compact formulas for the modified Macdonald polynomials, namely,
		$$\widetilde{H}_{\lambda}(X;q,t)=\sum_{\sigma}d_{\varepsilon}(\sigma)x^{\sigma}t^{\maj(\sigma)}q^{\eta(\sigma)},$$
		which are summed over all canonical or dual canonical fillings of a Young diagram, where $d_{\varepsilon}(\sigma)$ is a product of $q$-multinomials. 
		Finally, the compact formulas enable us to derive four explicit expressions for the monomial expansion of modified Macdonald polynomials, one of which coincides with the formula due to Garbali and Wheeler (2020), obtained by making extensive use of the Yang-Baxter integrability approach.
	\end{abstract}
	
	\maketitle
	
	\section{Introduction and main results}\label{S:intro}
	
	Macdonald polynomials $P_{\lambda}(X;q,t)$ indexed by partitions, are symmetric polynomials in infinitely many variables $X=(x_1,x_2,\ldots)$ with coefficients in the field $\mathbb{Q}(q,t)$. They are defined as the unique basis for the algebra of symmetric functions over $\mathbb{Q}(q,t)$, characterized by triangularity and orthogonality axioms \cite{Haiman99,Mac95}. The Macdonald polynomials $P_{\lambda}(X;q,t)$ unify several important families of symmetric functions, such as Hall--Littlewood functions ($q=0$), $q$-Whittaker functions ($t=0$), Jack functions (set $q=t^{\alpha}$ and let $t\rightarrow 1$), monomial symmetric functions ($q=0$ and $t=1$) and Schur functions ($q=t$).
	
	{\em Modified Macdonald polynomials} $\widetilde{H}_{\lambda}(X;q,t)$ are deduced from $P_{\lambda}(X;q,t)$ by a plethystic substitution, and are symmetric polynomials in $X$ with $\mathbb{N}[q,t]$-coefficients. The combinatorial interpretation of these coefficients was a celebrated breakthrough due to Haglund, Haiman and Loehr \cite{HHL04}, which connects the coefficients of $\widetilde{H}_{\lambda}(X;q,t)$ to permutation statistics on arbitrary fillings of Young diagrams. In the sequel, we represent a Young diagram in the French manner, and the precise definitions of some statistics are postponed to Section \ref{S:2}. 
	
	Let $\T(\lambda)$ denote the set of fillings of the Young diagram of $\lambda$ in which each cell is filled with a positive integer. Each filling is also called a tableau. 
	For $\tau\in \T(\lambda)$, let $x^\tau$ be the product of $x_i$'s whenever $i$ is an entry of $\tau$. Haglund, Haiman and Loehr \cite{HHL04} showed that
	\begin{align}\label{E:mmp1}
		\widetilde{H}_{\lambda}(X;q,t)=\sum_{\tau\in \T(\lambda)}x^{\tau}t^{\maj(\tau)}q^{\inv(\tau)}.
	\end{align}
	Here $\maj(T)$ and $\inv(T)$ are natural extensions, respectively, of the major index and inversion numbers for permutations. Recently, Corteel, Haglund, Mandelshtam, Mason and Williams\cite{CHO22} introduced a new statistic $\quinv$, called the queue inversion, and proposed a conjecture on an equivalent form of (\ref{E:mmp1}):
	\begin{align}\label{E:mmp2}
		\widetilde{H}_{\lambda}(X;q,t)=\sum_{\tau\in \T(\lambda)}x^{\tau}t^{\maj(\tau)}q^{\quinv(\tau)}.
	\end{align}
	This conjecture was subsequently  confirmed by Ayyer, Mandelshtam and Martin \cite{AMM23} by verifying that the right-hand side of (\ref{E:mmp2}) satisfies certain orthogonality and triangularity conditions which uniquely characterize the modified Macdonald polynomials $\widetilde{H}_{\lambda}(X;q,t)$. 
	
	The original motivation in \cite{CHO22} for studying the combinatorial formula of $\widetilde{H}_{\lambda}(X;q,t)$ stems from the search for new formulas with fewer summands and such formulas are named ``compact formulas''. One of their main results is a compact formula for $\widetilde{H}_{\lambda}(X;q,t)$ based on (\ref{E:mmp1}). Sooner thereafter, Mandelshtam \cite{O:23,O:24} found a compact version of (\ref{E:mmp2}). More precisely, both formulas reduce the sum on the right-hand side of (\ref{E:mmp1}) or (\ref{E:mmp2}) to the sum over 
	$\inv$-sorted (inversion-sorted) or $\quinv$-sorted (queue-inversion-sorted) tableaux. That is, for $\varepsilon\in \{\inv,\quinv\}$, 
	\begin{align}\label{E:sorted}
		\widetilde{H}_{\lambda}(X;q,t)=\sum_{\sigma:\, \varepsilon\text{-}\textrm{sorted}}\textrm{perm}(\sigma)x^{\sigma}t^{\maj(\sigma)}q^{\varepsilon(\sigma)}
	\end{align}
	summed over all $\varepsilon$-sorted tableaux, where $\textrm{perm}(\sigma)$ is a $q$-multinomial, counting the different ways to permute columns of equal height within a sorted tableau $\sigma$. 
	
	Independently, Garbali and Wheeler derived new formulas for $\widetilde{H}_{\lambda}(X;q,t)$ in terms of monomial symmetric functions through an integrable lattice model construction \cite{GW20}. This is indeed a sophisticated extension of the formula for $\widetilde{H}_{\lambda}(X;0,t)$ by Kerov, Kirillov and Reshetikhin \cite{KKR:88,K98,KKR:882}.
	
	Inspired by these recent developments, our research on the modified Macdonald polynomials $\widetilde{H}_{\lambda}(X;q,t)$ started from a refined equality of (\ref{E:mmp1}) and (\ref{E:mmp2}), which was conjectured by Ayyer, Mandelshtam and Martin \cite[Conjecture 10.3]{AMM23}. Two fillings $\sigma,\tau$ of the Young diagram of $\lambda$ are called {\em row-equivalent} if each row has the same multiset of entries in both fillings. To our surprise, the triples
	$(\maj,\inv,\quinv)$ and $(\maj,\quinv,\inv)$ have the same distribution over any equivalence class of a rectangular filling.
	\begin{theorem}\cite[Theorem 1.1]{JL24}\label{T:3}
		Let $\lambda$ be a partition, let $\sigma$ be a filling of the Young diagram of $\lambda$, and let $[\sigma]$ denote the row-equivalence class of $\sigma$. Then 
		\begin{align}\label{E:T1}
			\sum_{\tau\in [\sigma]}t^{\maj(\tau)}q^{\inv(\tau)}=\sum_{\tau\in [\sigma]}t^{\maj(\tau)}q^{\quinv(\tau)}.
		\end{align}	 
		If the diagram of $\lambda$ is a rectangle, then
		\begin{align}\label{E:T2}
			\sum_{\tau\in [\sigma]}t^{\maj(\tau)}q^{\inv(\tau)}u^{\quinv(\tau)}=\sum_{\tau\in [\sigma]}t^{\maj(\tau)}u^{\inv(\tau)}q^{\quinv(\tau)}.
		\end{align}
	\end{theorem}
    The main contribution of this paper is to bridge the gap between the combinatorial formulas (\ref{E:mmp1})--(\ref{E:sorted}) and explicit computable expressions for the modified Macdonald polynomials. As a first step, we introduce a family $\A$ of sixteen new statistics on fillings of a Young diagram, where each statistic in $\A$ plays the same role as the statistic inversion ($\inv$) in (\ref{E:T1}). This yields new combinatorial formulas for $\widetilde{H}_{\lambda}(X;q,t)$.
    \begin{theorem}\label{T:4}
    	Let $[\sigma]$ denote the row-equivalence class of $\sigma$, and let $\A$ be the family of statistics in Definition \ref{Def:eta}. Then for every $\eta\in \A$,
    	\begin{align}\label{E:eq16}
    		\sum_{\tau\in[\sigma]}t^{\maj(\tau)}q^{\quinv(\tau)}=\sum_{\tau\in[\sigma]}t^{\maj(\tau)}q^{\eta(\tau)}.
    	\end{align}
    	Consequently, the modified Macdonald polynomial is given by
    	\begin{align}\label{E:mmp3}
    		\widetilde{H}_{\lambda}(X;q,t)=\sum_{\tau\in \T(\lambda)}x^{\tau}t^{\maj(\tau)}q^{\eta(\tau)}.
    	\end{align}
    \end{theorem}
    	\begin{remark}
    	Loehr constructs explicit bijections  \cite[Theorem 1]{Loehr} for coinversion identities used in Ayyer--Mandelshtam--Martin's proof of the queue-inversion formula \cite{AMM23}. His correspondence in \cite[Equation (6)]{Loehr} changes the content of a word by replacing one occurrence of the smallest letter with the largest letter. In contrast, our bijection $\gamma$ to prove \eqref{E:eq16} preserves the multiset of entries in each row of a filling, and hence its content.
    \end{remark}

    Building on (\ref{E:mmp3}), in the second step we introduce canonical tableaux and dual canonical tableaux to present new compact formulas for the modified Macdonald polynomials.
   	\begin{theorem}\label{T:2}
   	Let $\CMcal C_\diamond(\lambda)$ and $\CMcal C_\dagger(\lambda)$ be the sets of canonical and dual canonical tableaux, respectively. For $\varepsilon\in\{\diamond,\dagger\}$ and each of the four statistics $\eta\in\A^+$ satisfying the corresponding condition in Theorem~\ref{P:can11},
   	\begin{align}\label{eq10}
   		\widetilde H_\lambda(X;q,t)=\sum_{\sigma\in\CMcal C_\varepsilon(\lambda)}
   		d_\varepsilon(\sigma)x^\sigma t^{\maj(\sigma)}q^{\eta(\sigma)},
   	\end{align}
   	where $d_\varepsilon(\sigma)$ is a product of $q$-multinomial coefficients. 
   \end{theorem}
    The canonical fillings are defined as special sorted tableaux appearing in \cite{CHO22,O:23}. An advantage of making use of canonical or dual canonical tableaux lies in the fact that (\ref{eq10}) can be further used to generate computable monomial expansions of the modified Macdonald polynomials. One of them is consistent with the formula given by Garbali and Wheeler \cite{GW20}. Throughout the paper, formulas are called {\em distinct} when their summands differ termwise. We use the following notation, adapted from \cite{GW20}. The Gaussian binomial coefficient is
    	\begin{align*}
    	{n\brack k}_q=\frac{[n]_q!}{[k]_q![n-k]_q!},
    \end{align*}
    for $n\ge k\ge 0$, where $[n]_q=1+q+\cdots+q^{n-1}=(1-q^n)/(1-q)$ for $n\ge 1$, $[0]_q=0$, $[0]!_q=1$ and $[n]!_q=[n]_q[n-1]_q\cdots [1]_q$. We set ${n\brack k}_q=0$ unless $0\le k\le n$. 
    
    For weakly increasing sequences of nonnegative integers $\nu=(\nu^1 \leq \cdots \leq \nu^N)$  and $\tilde{\nu}=(\tilde{\nu}^1 \leq \cdots \leq \tilde{\nu}^N)$, we set $\nu^0=\tilde\nu^0=0$ and define
	\begin{align}\label{eq5}
		\Phi_{\nu| \tilde{\nu}}(\xi,z;q)&=\sum_{s}q^{\xi(s,\tilde{\nu})} \prod_{1\le k<N} z^{ s_{N}^{k}-s_{k}^{k} }
		\begin{bmatrix} \tilde{\nu}^{k+1}-s_{k+1}^{1,k} \\ \tilde{\nu}^{k}-s_{k}^{1,k} \end{bmatrix}_{q} \prod_{1\le i\leq k} \begin{bmatrix} s_{k+1}^{i} \\ s_{k}^{i} \end{bmatrix}_{q}\in\mathbb{N}[z,q],
	\end{align}
	where the sum runs over arrays $\{s_k^i \}_{1\leq i\leq k\leq N}$ satisfying
	\begin{align}\label{E:defs}
		0\leq s_{k}^{k}\leq s_{k+1}^{k}\leq \cdots \leq s_{N-1}^{k}\leq s_{N}^{k}=\nu^{k}-\nu^{k-1},
	\end{align}
	with $s_{i}^{j,k}=\sum_{a=j}^{k}s_{i}^{a}$. Analogously, define
	\begin{align}\label{eq5d}
		\phi_{\nu| \tilde{\nu}}(\xi,z;q)&=\sum_{s}q^{\xi(s,\tilde{\nu})} \prod_{0\le k<N} z^{ s_{N}^{k}-s_{k}^{k} }
		\begin{bmatrix} \tilde{\nu}^{k+1}-s_{k+1}^{0,k} \\ \tilde{\nu}^{k}-s_{k}^{0,k} \end{bmatrix}_{q} \prod_{0\le i\leq k} \begin{bmatrix} s_{k+1}^{i} \\ s_{k}^{i} \end{bmatrix}_{q}\in\mathbb{N}[z,q],
	\end{align}
	where the sum runs over arrays $\{s_k^i \}_{0\le i<N,i\leq k\leq N}$ such that
	\begin{align}\label{E:defs2}
		0\leq s_{k}^{k}\leq s_{k+1}^{k}\leq \cdots \leq s_{N-1}^{k}\leq s_{N}^{k}=\nu^{k+1}-\nu^{k},
	\end{align}
	with $s_0^0=0$ and $s_{i}^{j,k}=\sum_{a=j}^{k}s_{i}^{a}$. 
	For $\gamma\in \{0,1\}$, let
	\begin{align*}
		\xi_{\gamma}(s,\tilde{\nu})&=\sum_{\gamma\leq i\leq k<N} (s_{k+1}^{i}-s_{k}^{i})  (\tilde{\nu}^{k}-s_{k}^{i,k}).
	\end{align*}
	For a partition $\lambda$, let $\lambda'$ be the conjugate of $\lambda$, that is, 
	the Young diagram of $\lambda'$ is obtained by interchanging the rows and columns of the Young diagram of $\lambda$. Let  
	\begin{align}\label{E:nlam}
		n(\lambda)=\sum_{i\ge 1}\binom{\lambda_i'}{2}=\sum_{i\ge 1}(i-1)\lambda_i.
	\end{align}
	\begin{theorem}\label{T:mono1}
		The modified Macdonald polynomial $\widetilde{H}_{\lambda}(X;q,t)$ equals
		\begin{align}\label{eq61}
			\widetilde{H}_{\lambda}(X;q,t)
			&=q^{n(\lambda')} \sum_{\mu\,\vdash |\lambda|}\CMcal{P}_{\lambda\mu}(q^{-1},t)m_{\mu}(X),\quad\mbox{ where }\\
			\label{eq6}\CMcal{P}_{\lambda\mu}(q,t)&=\sum_{\{\nu \}}q^{\chi_1(\nu)}  \prod_{1\leq i\leq j\leq n} \Phi_{\nu_{i+1,j}| \nu_{i,j}}  (\xi_1,t^{j-i}q^{\lambda_{i}-\lambda_{j}}; q),\\
			\label{eq62}&=\sum_{\{\nu \}}q^{\chi_2(\nu)}  \prod_{1\leq i\leq j\leq n} \Phi_{\nu_{i+1,j}| \nu_{i,j}}  (\xi_1,t^{i}q^{\lambda_{j-i+1}-\lambda_j}; q),\\
			\label{eq63}&=\sum_{\{\nu \}}t^{n(\lambda)}q^{n(\lambda')-\chi_3(\nu)}  \prod_{1\leq i\leq j\leq n} \phi_{\nu_{i+1,j}| \nu_{i,j}}  (\xi_0,t^{i-j}q^{\lambda_{j}-\lambda_i}; q^{-1}),\\
			\label{eq64}&=\sum_{\{\nu \}}t^{n(\lambda)}q^{n(\lambda')-\chi_4(\nu)}  \prod_{1\leq i\leq j\leq n} \phi_{\nu_{i+1,j}| \nu_{i,j}}  (\xi_0,t^{-i}q^{\lambda_{j}-\lambda_{j-i+1}}; q^{-1}).
		\end{align}
		The sums run over all weakly increasing sequences $\nu_{i,j}=(\nu_{i,j}^{1} \le \cdots \le \nu_{i,j}^{N})$ of nonnegative integers, $1\le i\le j\le n=\ell(\lambda)$, $N=\ell(\mu)$, with $\lambda_{n+1}=0$ and $\nu_{i,j}^{0}=0$, subject to
		\begin{align}
			\label{eq7}\nu_{i,j}^{N}=\lambda_{j}-\lambda_{j+1}, &\qquad \textrm{ for all }\,\,1\le i\le j\le n,   \\
			\label{eq8}\nu_{i+1,i}^{k}=0,     &\qquad \textrm{ for all }\,\, 1\le i\le n, 1\le k\le N, \\
			\label{eq9}\sum_{1\le i\le j\le n} \nu_{i,j}^{k}=\mu_1+\cdots+\mu_k, &\qquad \textrm{ for all }\,\, 1\le k\le N,
		\end{align}
		where
		\begin{align*}
			\chi_1(\nu)&=\sum_{k=1}^{N}\sum_{1\le i\le j\le n} \bigg( \binom{\nu_{i,j}^{k}-\nu_{i,j}^{k-1}}{2}+ \sum_{\ell>j}(\nu_{i,j}^{k}-\nu_{i,j}^{k-1})(\nu_{i,\ell}^{k}-\nu_{i+1,\ell}^{k-1}) \bigg),\\
			\chi_2(\nu)&=\sum_{k=1}^{N}\sum_{1\le i\le j\le n} \bigg( \binom{\nu_{i,j}^{k}-\nu_{i,j}^{k-1}}{2}+ \sum_{\ell>j}(\nu_{i,j}^{k}-\nu_{i,j}^{k-1})(\nu_{\ell-j+i,\ell}^{k}-\nu_{\ell-j+i+1,\ell}^{k-1}) \bigg),\\
			\chi_3(\nu)&=\sum_{k=1}^{N}\sum_{1\le i\le j\le n} \sum_{\ell>j}(\nu_{i,j}^{k}-\nu_{i,j}^{k-1})(\nu_{i,\ell}^{k-1}-\nu_{i+1,\ell}^{k}),\\
			\chi_4(\nu)&=\sum_{k=1}^{N}\sum_{1\le i\le j\le n} \sum_{\ell>j}(\nu_{i,j}^{k}-\nu_{i,j}^{k-1})(\nu_{\ell-j+i,\ell}^{k-1}-\nu_{\ell-j+i+1,\ell}^{k}).
		\end{align*}
	\end{theorem}
	Formula \eqref{eq6} recovers the manifestly positive expansion by Garbali and Wheeler \cite[Theorem 2]{GW20}. The four formulas are termwise distinct, as the next two examples illustrate.
	\begin{example}\label{Example:1}
		For $\lambda=(3,2)$ and $\mu=(4,1)$, we have $n=N=2$. The possible arrays are determined by $(\nu_{1,1}^1,\nu_{1,2}^1,\nu_{2,2}^1)$, since their terminal entries are $(\nu_{1,1}^2,\nu_{1,2}^2,\nu_{2,2}^2)=(1,2,2)$. The contributions of the three possible choices $\Omega_1$, $\Omega_2$ and $\Omega_3$ of $\{\nu\}$ to $\CMcal P_{\lambda\mu}(q,t)$ are
		\[
		\begin{array}{c|c|c|c}
			& (\nu_{1,1}^1,\nu_{1,2}^1,\nu_{2,2}^1)&\eqref{eq6}&\eqref{eq62}\\\hline
			\Omega_1&(0,2,2)&q^2&q^4\\
			\Omega_2&(1,1,2)&q^3t(1+q)&q^3t(1+q)\\
			\Omega_3&(1,2,1)&q^3(1+q)&q^2(1+q)
		\end{array}
		\]
		Both columns sum to $P_{\lambda\mu}(q,t)=q^2+q^3+q^4+t(q^3+q^4)$, but their summands differ.
	\end{example}
	\begin{example}\label{Example:2}
		For $\lambda=\mu=(4,2)$, we have $n(\lambda)=2$ and $n(\lambda')=7$. The terminal entries are $(\nu_{1,1}^2,\nu_{1,2}^2,\nu_{2,2}^2)=(2,2,2)$, and the contributions of $\Omega_i$ for $1\le i\le 6$  to $\CMcal P_{\lambda\mu}(q^{-1},t)$ are
		\[
		\begin{array}{c|c|c|c}
			&(\nu_{1,1}^1,\nu_{1,2}^1,\nu_{2,2}^1)&\eqref{eq63}&\eqref{eq64}\\\hline
			\Omega_1&(0,2,2)&q^{-3}&q^{-3}\\
			\Omega_2&(1,1,2)&q^{-4}(1+q^{-1})^2&q^{-3}(1+q^{-1})^2\\
			\Omega_3&(1,2,1)&tq^{-4}(1+q^{-1})^2&tq^{-4}(1+q^{-1})^2\\
			\Omega_4&(2,2,0)&t^2q^{-7}&t^2q^{-7}\\
			\Omega_5&(2,0,2)&q^{-7}&q^{-7}\\
			\Omega_6&(2,1,1)&q^{-3}(1+tq^{-3})(1+q^{-1})&q^{-5}(1+tq^{-1})(1+q^{-1})
		\end{array}
		\]
		The total column sums agree, although the summands differ.
	\end{example}
	
	The rest of the paper is organized as follows. In Section \ref{S:2} we provide preliminaries on the modified Macdonald polynomials and the combinatorial statistics $\inv,\quinv,\maj$ on fillings. The family $\A$ of new statistics is introduced in Section \ref{S:familyA}.
	In Sections \ref{S:3} and \ref{S:5}, the flip operators and the proof of Theorem \ref{T:4} are presented. Sections \ref{S:can}, \ref{S:thm3} and \ref{S:thm41} are devoted to bringing in (dual) canonical tableaux and completing the proof of Theorems \ref{T:2} and \ref{T:mono1}. Finally, we make some remarks in Section \ref{S:final_remark}.

	
	{\color{blue} Declaration on the use of AI}: All mathematical proofs presented in this article originate from the authors' own work. The authors used ChatGPT (OpenAI, GPT-6 Astra) for proofreading and English language polishing. The authors take full responsibility for accuracy and content of the manuscript.

	\section{Preliminaries and notation}\label{S:2}
	
	A partition $\lambda=(\lambda_1\ge \cdots\ge \lambda_l>0)$ of a positive integer $n$ is a weakly decreasing sequence of positive integers such that $\lambda_1+\cdots+\lambda_{l}=n$, denoted $\lambda\vdash n$. 
	The size and length of $\lambda$ are denoted by $|\lambda|=n$ and $\ell(\lambda)=l$, respectively. 
	
	We use the French convention to represent Young diagrams, i.e., the Young diagram $\dg(\lambda)$ has $\lambda_i$ left-justified boxes in row $i$, counted from the bottom.  A box has coordinates $(i,j)$ if it is in the $i$th row from the bottom and the $j$th column from the left. Let $\lambda'$ be the conjugate partition of $\lambda$, that is, $\dg(\lambda')$ is obtained from $\dg(\lambda)$ by reflecting across the main diagonal (the boxes with coordinates $(i,i)$). The {\em\color{blue} arm} and {\em \color{blue} leg} of a box $u$, denoted by $\arm(u)$ and $\leg(u)$, are the numbers of boxes strictly to its east and north, respectively.
	
	Let $\Lambda$ be the algebra of symmetric functions in infinitely many variables $X=(x_1,x_2,\ldots)$, with coefficients in $\mathbb{Q}(q,t)$. The bases of $\Lambda$ are indexed by partitions, such as the {\em monomial symmetric functions} $m_{\mu}(X)$, the {\em elementary symmetric functions} $e_{\mu}(X)$, the {\em complete homogeneous symmetric functions} $h_{\mu}(X)$, the {\em power-sum symmetric functions} $p_{\mu}(X)$, and the {\em Schur functions} $s_{\mu}(X)$; for instance, see the books \cite{Hag08,Mac95, Stanley}. 
	
	For partitions $\lambda$ and $\mu$ of equal size, we write
	$\lambda\leq\mu$ {\em \color{blue} in dominance order} if $\sum_{i=1}^{r}\lambda_i\leq\sum_{i=1}^{r}\mu_i$
	for every $r$ with $1\le r\le\min\{\ell(\lambda),\ell(\mu)\}$. The {\em \color{blue} scalar product} $\langle\cdot,\cdot\rangle$ of symmetric functions is defined on the power-sum basis by
	\begin{align*}
		\langle p_{\lambda},p_{\mu}\rangle_{q,t}=z_{\lambda}\delta_{\lambda\mu}\prod_{i=1}^{\ell(\lambda)}\frac{1-q^{\lambda_i}}{1-t^{\lambda_i}},
	\end{align*}
	where $z_{\lambda}=1^{m_1}m_1!2^{m_2}m_2!\cdots$ and $m_i=m_i(\lambda)$ denotes the number of parts of $\lambda$ equal to $i$. The {\em \color{blue} Macdonald polynomial} $P_{\mu}(X,q,t)$ is defined as the unique symmetric function satisfying both of the following conditions:
	\begin{align*}
		P_{\mu}(X;q,t)=m_{\mu}(X)+\sum_{\lambda<\mu}a_{\lambda\mu}(q,t)m_{\lambda}(X)
	\end{align*}
	with coefficients $a_{\lambda\mu}(q,t)\in\mathbb{Q}(q,t)$, and $\langle P_{\lambda},P_{\mu}\rangle=0$ if $\lambda\ne \mu$. Subsequently, Macdonald introduced the {\em \color{blue} integral Macdonald polynomial}
	\begin{align*}
		J_{\mu}(X;q,t)=\prod_{s\in\dg(\mu)}(1-q^{\mathsf{arm}(s)}t^{1+\mathsf{leg}(s)})
		P_{\mu}(X;q,t).
	\end{align*}
To define the modified Macdonald polynomials, we will use the plethystic notations. For any polynomial or formal series $A$ in the indeterminates $a_1,a_2,\ldots$, we define $p_k[A]$ by replacing each indeterminate $a_i$ in $A$ by $a_i^k$. For every $f\in \Lambda$, the {\em \color{blue} plethystic substitution} $f[A]$ is obtained by expressing $f$ as a polynomial in the power-sum symmetric functions $p_k$ and replacing $p_k$ by $p_k[A]$. For instance, let $X=\sum_{i\ge 1}x_i$, then  
\begin{align*}
p_k[X(1-t)]=p_k\left[\sum_{i\ge 1}(x_i-tx_i)\right]=\sum_{i\ge 1}(x_i^k-t^kx_i^k)=(1-t^k)p_k[X].
\end{align*}
The Schur function $s_{\lambda}[X(1-t)]$ is called the {\em \color{blue} plethystic Schur function}. Expanding $J_{\mu}(X;q,t)$ in the plethystic Schur basis $s_{\lambda}[X(1-t)]$ gives
\begin{align*}
	J_{\mu}(X;q,t)=\sum_{\lambda\,\vdash |\mu|}K_{\lambda\mu}(q,t)s_{\lambda}[X(1-t)].
\end{align*}
Macdonald \cite{Mac88,Mac95} conjectured that $K_{\lambda\mu}(q,t)$ are polynomials in $q,t$ with nonnegative coefficients. This conjecture was resolved by Haiman \cite{Hag08,Haiman01}, who showed that the modified Macdonald polynomial $\widetilde{H}_{\lambda}(X;q,t)$ equals the Frobenius series of a space defined as the linear span of certain polynomials and their all partial derivatives \cite{Haiman01}. The {\em \color{blue} modified Macdonald polynomial} $\widetilde{H}_{\lambda}(X;q,t)$ was introduced by Garsia and Haiman \cite{GH93,GH96}: 
\begin{align*}
	\widetilde{H}_{\mu}(X;q,t)
	=t^{n(\mu)} J_{\mu}\left[\frac{X}{1-t^{-1}};q,t^{-1}\right]=t^{n(\mu)}\sum_{\lambda}K_{\lambda\mu}(q,t^{-1})s_{\lambda}(X),
\end{align*}
where $n(\mu)$ is defined in (\ref{E:nlam}). Since $K_{\lambda\mu}(q,t)\in\mathbb{N}[q,t]$ and $K_{\lambda\mu}(q,t)$ has degree at most $n(\mu)$ in $t$ by \cite[Chapter VI, (8.14)]{Mac95}, one immediately gets $t^{n(\mu)}{K}_{\lambda\mu}(q,t^{-1})\in\mathbb{N}[q,t]$. Various formulas for $K_{\lambda\mu}(q,t)$ and its integrality have been extensively investigated; see for instance \cite{GHT96,GR98,GT96,KN98,Knop97,Sahi96}. 

	\begin{example} For $\lambda=(2,2)$,
		\begin{align*}
			\widetilde{H}_{22}(X;q,t)&=s_4(X)+(q+t+qt)s_{31}(X)+(t^2+q^2)s_{22}(X)\\
			&\qquad+(qt^2+q^2t+qt)s_{211}(X)+q^2t^2s_{1111}(X).
		\end{align*}
	\end{example}
	We next present the combinatorial formulas for $\widetilde{H}_{\lambda}(X;q,t)$. Throughout the paper, we write $\#A$ for the cardinality of a finite set, $[n]=\{1,\ldots,n\}$, and $\chi(S)$ for the indicator of a statement $S$. A {\em \color{blue} filling (or tableau)} is a map $\sigma:\dg(\lambda)\to\mathbb P$, where $\mathbb P$ denotes the positive integers. 
	We write $\sigma(i,j)$ for its entry in box $(i,j)$ and set
	\[
	x^\sigma=\prod_{u\in\dg(\lambda)}x_{\sigma(u)}.
	\]
	For a box $u$ above the bottom row, let $\South(u)$ be the box immediately below it. The pair $(\sigma(u),\sigma(\South(u)))$ is a \emph{\color{blue}descent} if $\sigma(u)>\sigma(\South(u))$, and a \emph{\color{blue} non-descent} otherwise. We sometimes call such a pair {\em \color{blue}a vertical pair}. Define
	\[
	\Des(\sigma)=\{u\in\dg(\lambda):\South(u)\in\dg(\lambda),\
	\sigma(u)>\sigma(\South(u))\}
	\]
	and $\des(\sigma)=\#\Des(\sigma)$.
	The \emph{\color{blue} major index} of a filling $\sigma$ is
	\[
	\maj(\sigma)=\sum_{u\in\Des(\sigma)}(\mathrm{leg}(u)+1).
	\]
	\begin{definition}\cite{AMM23,CHO22} \label{Def:quinv} ({\em \color{blue} Queue inversion})
		Let $\T(\lambda)$ denote the set of all fillings of $\mathrm{dg}(\lambda)$.
		Given a filling $\sigma\in \T(\lambda) $, let
		$\sigma(\lambda'_{i}+1,i)=0$ for all $i$, thus adjoining a box containing $0$ above each column. A {\em \color{blue} queue inversion triple} of $\sigma$ is a triple $(a,b,c)$ of entries in $\sigma$ satisfying the following conditions (Figure~\ref{F:f1}, left; $a$ may be zero):
		\begin{enumerate}
			\item $b$ and $c$ are in the same row and $c$ is to the right of $b$;
			\item $a$ and $b$ are in the same column such that $b$ is immediately below $a$;
			\item one of the conditions $a<b<c$, $b<c<a$, $c<a<b$ or $a=b\ne c$ holds.
		\end{enumerate}
		Let $\quinv(\sigma)$ be the number of queue inversion triples of $\sigma$.
	\end{definition}
	\begin{figure}[ht]
		\centering
		\includegraphics[scale=0.4]{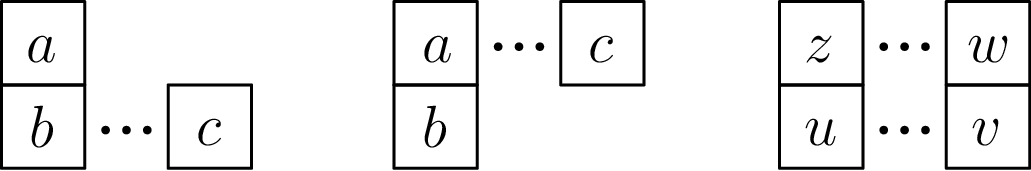}
		\caption{A queue inversion triple (left), an inversion triple (middle) and a quadruple (right).}\label{F:f1}
	\end{figure}
	
	\begin{definition}\cite{HHL04} ({\em \color{blue} Inversion}) \label{Def:inv}
		Define $\sigma(0,i)=\infty$ for each column $i$ of $\sigma$, thus adjoining a box containing $\infty$ below each column. An {\em \color{blue} inversion triple} of $\sigma$ is a triple $(a,b,c)$ of entries in $\sigma$ satisfying conditions (2)--(3) above and (4) below (Figure~\ref{F:f1}, middle; $b$ may be infinity).
		\begin{enumerate}
			\setcounter{enumi}{3}
			\item $a$ and $c$ are in the same row and $c$ is to the right of $a$.
		\end{enumerate}
		Let $\inv(\sigma)$ be the number of inversion triples of $\sigma$.
		
		For any ordered triple of values, set $\Q(a,b,c)=1$ when one of the inequalities in condition (3) holds, and $\Q(a,b,c)=0$ otherwise. 
		Note that 
		\begin{align}\label{E:Qindicator}
			\Q(a,b,c)=\chi(b<c)+\chi(a\le b)-\chi(a\le c).
		\end{align}
		The identity can be verified directly by checking all possible orderings of $a,b,c$. 
	\end{definition}

	\begin{remark}\label{Rem:coinv}
		The reading word $w(\sigma)$ lists the entries from left to right in each row, starting with the top row. It is a sequence, possibly with repeated letters. For a one-row filling $\sigma$ with reading word $w(\sigma)=w_1\cdots w_n$, we have
		\[
		\inv(\sigma)=\#\{i<j:w_i>w_j\},\qquad
		\quinv(\sigma)=\#\{i<j:w_i<w_j\},
		\]
		where $\quinv(\sigma)=\coinv(w)$ is called the number of coinversions.
		For a one-column filling $\sigma$, $\des(\sigma)$ and $\maj(\sigma)$ are the descent number and major index of its reading word $w(\sigma)$, respectively. These specialize to the usual permutation statistics.
	\end{remark}
	\section{A family of new statistics}\label{S:familyA}
	This section is contributed to introducing new statistics $\eta$. 
	We define eight statistics via local conditions on quadruples and obtain eight additional ones by complementing and reflecting each rectangular component, for a total of sixteen.
	\begin{definition}[{\em \color{blue}$\quinv$-quadruple sets}]\label{Def:typeqinv}
		Let $\S$ be a set of ordering patterns for quadruples $(z,w,u,v)$, whose elements are determined by the following rules.
		\begin{enumerate}
			\item If $z\ne w$, then $(z,w,u,v)\in\S$ if and only if $(w,z,v,u)\in\S$. Otherwise $z=w$, then $(z,w,u,v)\in\S$ if and only if $\Q(z,u,v)=1$. That is, $z=w>v>u$, $u\ge z=w>v$, and $v>u\ge z=w$.
			\item No quadruple with $u=v$ belongs to $\S$. If $u\ne v$, then exactly one of $(z,w,u,v)$ and $(z,w,v,u)$ belongs to $\S$.
			\item The patterns $z>u>v\ge w$ and $v\ge z>w>u$ are either both included or both excluded.
			\item Both patterns $z>v\ge w>u$ and $u\ge z>v\ge w$ are included.
		\end{enumerate}
		A set $\S$ satisfying these requirements is called a \emph{\color{blue} $\quinv$-quadruple set}.
	\end{definition}
	
	\begin{example}\label{Example:S}
		One such set $\S$ contains the three patterns 
		\begin{align*}
			z=w>v>u,\qquad u\ge z=w>v,\qquad v>u\ge z=w
		\end{align*}
	    in (1) when $z=w$, the six patterns
		\begin{align*}
			&z>u>v\ge w,\qquad z>v\ge w>u,\qquad v\ge z>w>u,\\
			&u\ge z>v\ge w,\qquad v>u\ge z>w,\qquad z>w>v>u,
		\end{align*}
		when $z>w$, and their images under $(z,w,u,v)\mapsto(w,z,v,u)$. 
	\end{example}
	We associate each $\quinv$-quadruple set $\S$ with a statistic $\eta$ that can replace $\quinv$ or $\inv$ in the formulas \eqref{E:mmp1} and \eqref{E:mmp2}.
	\begin{definition}[{\em \color{blue} Family $\A$ of statistics}]\label{Def:eta}
		For $\sigma\in\T(\lambda)$, adjoin a zero above each column of $\sigma$. Consider quadruples $(z,w,u,v)$ in the configuration shown on the right of Figure~\ref{F:f1}, with $(z,u)$ and $(w,v)$ in two columns of equal height. Let $\Q_{\S}(z,w,u,v)$ be the indicator that their ordering pattern belongs to $\S$, and set
		\begin{align*}
			\eta_{\S}(\sigma)=\sum_{(z,w,u,v)\in\sigma}\Q_{\S}(z,w,u,v).
		\end{align*}
		Let $\Q_{\S}(a,b,c)$ indicate that $(a,b,c)$ is a queue inversion triple and the column containing $(a,b)$ is longer than the column containing $c$. Define
		\begin{align}\label{E:etai}
			\eta(\sigma)=\eta_{\S}(\sigma)+\sum_{(a,b,c)\in\sigma}\Q_{\S}(a,b,c),
		\end{align}
		and let $\A^+$ consist of the statistics $\eta$ obtained from (\ref{E:etai}) for all $\quinv$-quadruple sets $\S$. See Example \ref{Example:3.3}.
		
		For each $\eta\in\A^+$, we now define its counterpart $\eta^*$. Decompose the diagram into maximal rectangles such that their heights are strictly decreasing from the left \cite{CHO22}, and write
		\begin{align}\label{E:sigmadec}
			\sigma=\sigma_n\sqcup\cdots\sqcup\sigma_1,
		\end{align}
		where $n=\ell(\lambda)$ and $\sigma_j$ is the possibly empty rectangle of height $j$; see Example \ref{Example:3.2}. Fix a set $[N]=\{1,\ldots, N\}$ containing all entries of a filling $\sigma$. In each rectangle $\sigma_j$ of (\ref{E:sigmadec}), replace $x$ by $x^c=N+1-x$ and reverse the order of the rows. Denote this involution by $\sigma\mapsto\sigma'$. Adjoin $\infty$ below every column of $\sigma'$ and define
		\begin{align}\label{E:eta*}
			\eta^*(\sigma')=\eta_{\S}(\sigma)+\sum_{(a,b,c)\in\sigma'}\Q^*_{\S}(a,b,c),
		\end{align}
		where $\Q^*_{\S}(a,b,c)$ indicates an inversion triple whose column containing $(a,b)$ is longer than the column containing $c$. Finally, set
		\begin{align*}
			\A=\A^+\,\dot\cup\,\{\eta^*: \eta\in\A^+\}.
		\end{align*}

		
		
	\end{definition}
		\begin{remark}\label{Rem:3}
		By Definition \ref{Def:typeqinv}, we have $\Q_{\S}(z,w,u,v)=\Q_{\S}(w,z,v,u)$ for $z\ne w$; both sides vanish when $u=v$. For a rectangular filling $\sigma$, \eqref{E:etai}--\eqref{E:eta*} give $\eta^*(\sigma')=\eta(\sigma)$. However, this equality does not hold for a nonrectangular filling.
	\end{remark}

	\begin{example}\label{Example:3.2}
		For the filling $\sigma$ below,  the first two columns form
		the rectangle $\sigma_3$ (in green), and the last two columns form
		the rectangle $\sigma_2$; the rectangle $\sigma_1$ is empty.
		Thus $\sigma=\sigma_3\sqcup\sigma_2$. We now construct $\sigma'$ according to Definition \ref{Def:eta}. Take $N=6$. 
		First replace each entry $x$ by
		$7-x$, and then reverse the order of the rows within
		each rectangle:
		\[
		\sigma=
		\begin{ytableau}
			*(green)1 & *(green)4 \\
			*(green)3 & *(green)6 & 2 & 5 \\
			*(green)5 & *(green)2 & 4 & 1
		\end{ytableau}
		\ \xrightarrow{\;x\mapsto 7-x\;}\
		\begin{ytableau}
			6 & 3 \\
			4 & 1 & 5 & 2 \\
			2 & 5 & 3 & 6
		\end{ytableau}
		\ \xrightarrow{\;\text{row reversal}\;}\
		\begin{ytableau}
			2 & 5 \\
			4 & 1 & 3 & 6 \\
			6 & 3 & 5 & 2
		\end{ytableau}
		=\sigma'.
		\]
		Applying the same transformation to $\sigma'$ with the same $N=6$ recovers $\sigma$.
	\end{example}

	\begin{example}\label{Example:3.3}
		For the set $\S$ as in Example~\ref{Example:S} and a filling
		\[
		\sigma=\begin{ytableau}0&0&0\\5&4&5&0\\3&6&3&1\\1&7&2&8\end{ytableau}\,,
		\]
		we have $\eta_{\S}(\sigma)=4$, with contributing quadruples $(0,0,4,5)$, $(5,4,3,6)$, $(4,5,6,3)$, and $(3,3,1,2)$, while $\eta(\sigma)=6$, and the additional contributions arising from the unequal-height triples are $(4,6,1)$ and $(6,7,8)$. 
	\end{example}
	
	\section{Flip operators and statistics}\label{S:3}
	
	The purpose of this section is twofold: to define the flip operators from \cite{AMM23} that act on two neighboring columns of equal height, and to introduce new ones tailored for the family $\A^+$ of statistics $\eta$. We establish the properties of these operators needed in the subsequent proofs, and adopt the notation from \cite{AMM23,O:24}.
	
	\begin{definition}[{\em \color{blue}Flip operators}]\label{Def:2}
		Suppose that adjacent columns $i,i+1$ have the same height $h=\lambda'_i=\lambda'_{i+1}$ in the Young diagram $\dg(\lambda)$. For a filling $\sigma\in \T(\lambda)$ and $1\le r\le h$, let $t_i^r$ be the operator acting on $\sigma$ by interchanging the entries $\sigma(r,i)$ and $\sigma(r,i+1)$. For $1\le k\le r\le \lambda'_i$, let
		\begin{align*}
			t_i^{[k,r]}=t_i^{k}\circ t_i^{k+1}\cdots \circ t_i^{r}
		\end{align*}
		denote the operator that swaps entries of boxes $(x,i)$ and $(x,i+1)$ for all $k\le x\le r$. As before, adjoin zeros above the columns and infinities below them.
		
		The \emph{\color{blue} $\quinv$-flip operator} is $\rho_i^r=t_i^{[k,r]}$, where $k$ is the largest integer such that $1\le k\le r$ and 
		\begin{align}\label{E:cri}
			\Q(\sigma(k,i),\sigma(k-1,i),\sigma(k-1,i+1))
			=\Q(\sigma(k,i+1),\sigma(k-1,i),\sigma(k-1,i+1)).
		\end{align}
		Such a $k$ exists because the equality holds at $k=1$. The \emph{\color{blue}$\A$-flip operator} is $\delta_i^r=t_i^{[k,r]}$, where $k$ is the largest integer such that $1\le k\le r$ and $\sigma(k,i)=\sigma(k,i+1)$; if no such integer exists, we set $k=1$. Write $\rho_i=\rho_i^h$ and $\delta_i=\delta_i^h$. Both operators are involutions.
	\end{definition}
    \begin{remark}
    	The $\quinv$-flip operator $\rho_i^r$ first appeared in \cite{AMM23} and was used to prove (\ref{E:mmp2}). This operator was inspired by the $\inv$-flip operator of \cite{LN12}, which was introduced by Loehr and Niese to establish, when $t$ is specialized to a root of unity, a factorization of the modified Macdonald polynomials by combinatorial means.
    \end{remark}
		
The lemma below extends the properties of $\rho_i$ to $\rho_i^r$, which will be used to construct the desired bijection for Theorem \ref{T:4}. We note that the $\quinv$-flip operator is defined slightly different in \cite{JL24}; nevertheless, the same conclusion holds under either definition.	
\begin{lemma}\cite[Lemma 5.3]{JL24}\label{L:2}
Given a partition $\lambda$ and an integer $i$ such that $\lambda'_i=\lambda'_{i+1}$, let $\sigma\in\T(\lambda)$. 

Suppose $\raisebox{-6pt}{\young(ab,cd)}$ is part of $\sigma$ consisting of rows $r,r+1$ and columns $i,i+1$. If $\CMcal{Q}(a,c,d)=\CMcal{Q}(b,c,d)$, then
		\begin{align}\label{E:maj11}
			\maj(\sigma)&=\maj(\rho_i^r(\sigma)),\\
			\label{E:quinv1}
			\quinv(\sigma)&=\quinv(\rho_i^r(\sigma))-\Q(a,d,c)+\Q(a,c,d).
		\end{align}
		That is, the number of queue inversion triples induced between column $j$ and either column $i$ or column $i+1$ of $\sigma$, for any $j>i+1$, remains invariant under $\rho_i^r$.
\end{lemma}
For a filling $\tau$, let $\tau|_k^r$ denote its restriction
to rows $k,k+1,\ldots, r$. 
	%

	
	
	\begin{lemma}\label{L:delta}
		Suppose that columns $i$ and $i+1$ of $\sigma$ have
		the same height $h=\lambda'_i=\lambda'_{i+1}$ in the Young diagram $\dg(\lambda)$. Then, for $1\leq r\leq h$, we define $\widehat{\sigma}=\delta_i^r(\sigma)$ and obtain 
		\begin{align}
			\maj(\sigma)-\maj(\widehat{\sigma})
			&
			=(h-r)\left(
			\des(\sigma|_r^{r+1})
			-\des(\widehat{\sigma}|_r^{r+1})
			\right),
			\label{E:maj1}\\
			\eta_{\S}(\sigma)-\eta_{\S}(\widehat{\sigma})
			&=\eta_{\S}(\sigma\vert_{r}^{r+1})-\eta_{\S}(\widehat{\sigma}\vert_{r}^{r+1})\label{E:eta_S1}.
		\end{align}
	\end{lemma}
	
	\begin{proof}
		Write $\widehat{\sigma}=t_i^{[k,r]}(\sigma)$.
		The assertion is immediate for the pairs of consecutive
		rows lying entirely above or below rows $k,k+1,\ldots,r$.
		If $k>1$, then the two entries in row $k$ are equal, i.e., $\sigma(k,i)=\sigma(k,i+1)$, by
		the definition of $\delta_i^r$. Hence row $k$ is
		unchanged, and so are the contributions to the statistics between
		rows $k-1$ and $k$. 
		
		Now let $k\leq s<r$. In rows $s$ and $s+1$,
		the two vertical pairs in columns $i,i+1$ exchange
		places, preserving their total number of descents.
		Moreover, the choice of $k$ implies that
		\[
		\sigma(s+1,i)\neq\sigma(s+1,i+1).
		\]
		Definition~\ref{Def:typeqinv} (1) guarantees that the number of $\S$-quadruples formed by these
		two columns is invariant. Since the two columns $i$ and $i+1$ are adjacent, the
		two contributions involving any third column of
		height $h$ to $\S$-quadruples are exchanged. It follows that only the contributions between rows $r$ and
		$r+1$ can change the major index and the number of $\S$-quadruples, thus obtaining (\ref{E:eta_S1}). Moreover, any descent changed by the flip in one of the two
		columns of height $h$ contributes $h-r$ to the
		major index, yielding (\ref{E:maj1}).
	\end{proof}
	
		\begin{lemma}\label{L:crossblock}
		For any partition $\lambda$, let $f:\T(\lambda)\rightarrow \T(\lambda)$ be a map that preserves 
		\begin{itemize}
			\item the multiset of entries in every row of every rectangular component, and
			\item the number of descents between each pair of consecutive rows within every rectangular component. 
		\end{itemize}
		 Then $f$ also preserves the number of queue inversion triples (resp. inversion triples) that are located in two rectangles of unequal heights.
	\end{lemma}
	\begin{proof}
		Let $\sigma\in T(\lambda)$, then \eqref{E:Qindicator} implies 
		\begin{align*}
			\sum_{(a,b,c)\in\sigma}\Q_{\S}(a,b,c)=\sum_{(a,b,c)\in\sigma}(\chi(b<c)-\chi(a\le c))+\sum_{(a,b,c)\in\sigma}\chi(a\le b).
		\end{align*}
		Note that the first sum on the right-hand side depends only on the row multisets, and the second sum equals the product of the number of non-descents between two consecutive rows of $a,b$ and the width of the rectangle containing $c$. Therefore, the sum on the left-hand side is preserved by the map $f$. The statement for inversion triples follows similarly.

	\end{proof}
	
	\begin{lemma}\label{L:top}
		Under the assumption stated in Lemma \ref{L:delta}, for every $\eta\in\A^+$, we have
		\begin{align}\label{E:top1}
			\quinv(\sigma)-\quinv(\rho_i(\sigma))&=\eta(\sigma)-\eta(\delta_i(\sigma)),\\
			\label{E:topmaj1}
			\maj(\rho_i(\sigma))&=\maj(\delta_i(\sigma))=\maj(\sigma).
		\end{align}
	\end{lemma}
	\begin{proof}
		Applying Lemmas~\ref{L:2} and \ref{L:delta} with $r=h$ and $(a,b)=(0,0)$, we obtain (\ref{E:topmaj1}). By Lemma~\ref{L:crossblock}, the unequal-height contributions of queue inversion triples to the statistic $\eta$ are invariant under the map $\delta_i$. Therefore, we have 
		\begin{align*}
			\eta(\sigma)-\eta(\delta_i(\sigma))&=\eta_{\S}(\sigma)-\eta_{\S}(\delta_i(\sigma))\\
			&=\Q_{\S}(0,0,c,d)-\Q_{\S}(0,0,d,c)\quad \mbox{ by }\, (\ref{E:eta_S1}),\\
			&=\Q(0,c,d)-\Q(0,d,c)\quad\mbox{ by Definition}\, \ref{Def:typeqinv}\,(1),\\
			&=\quinv(\sigma)-\quinv(\rho_i(\sigma))\quad \mbox{ by }\, (\ref{E:quinv1}).
		\end{align*}
        This concludes the proof of (\ref{E:top1}).
	\end{proof}

We point out a major difference between the operators $\rho_i^r$ and $\delta_i^r$: the braid relation. The operators $\delta_i^r$ satisfy the braid relation (see Lemma \ref{L:braid}), whereas $\rho_i^r$ do not (see \cite{CHO22}).

	\begin{lemma}\label{L:braid}
		If columns $i,i+1,i+2$ have the same height $h$ in the diagram $\dg(\lambda)$, then
		\[
		\delta_i^r\delta_{i+1}^r\delta_i^r
		=\delta_{i+1}^r\delta_i^r\delta_{i+1}^r
		\]
		holds for all $1\le r\le h$.
	\end{lemma}
	\begin{proof}
		Without loss of generality, it suffices to consider the rectangular diagram $\dg(\lambda)$ of $r$ rows and three columns. Since the operators $\delta_i$ are involutions, we only need to show
		\begin{align}\label{L1}
			(\delta_1\delta_2)^3=\mathrm{id}.
		\end{align}
        Let $\sigma\in \T(\lambda)$ be a filling. If the entries in the top row of $\sigma$ are distinct, then $(\delta_1\delta_2)^3(\sigma\vert_r^r)=\sigma\vert_r^r$ and the proof of $(\delta_1\delta_2)^3(\sigma)=\sigma$ reduces to establishing $(\delta_1\delta_2)^3(\sigma\vert_1^{r-1})=\sigma\vert_1^{r-1}$. Therefore, we will prove $(\delta_1\delta_2)^3(\sigma)=\sigma$ when the top row of $\sigma$ has repeated entries.
        	If all three top-row entries of $\sigma$ are equal, then both operators $\delta_1$ and $\delta_2$ act as the identity. If exactly two are equal, say $(a,b,c)$ with $a=b\ne c$, then     	
        		\begin{align*}
        		\young(aac) \xrightarrow {\delta_2} \young(aca) \xrightarrow {\delta_1} \young(caa) \xrightarrow {\delta_2} \young(caa) \xrightarrow {\delta_1} \young(aca) \xrightarrow {\delta_2} \young(aac) \xrightarrow {\delta_1} \young(aac)\,,
        	\end{align*}
          where the second $\delta_2$ and the last $\delta_1$ are the identity. It follows that 
          \begin{align*}
          	\delta_1\circ \delta_2 \circ \delta_1\circ \delta_2 \circ \delta_1\circ \delta_2(\sigma)=
          	\delta_2 \circ \delta_1\circ \delta_1\circ \delta_2(\sigma)=\sigma.
          \end{align*}
          The other two cases, $a=c\ne b$ or $a\ne b=c$, follow analogously. This proves \eqref{L1}.
          \end{proof}

	\section{Proof of Theorem \ref{T:4}}\label{S:5}
    This section presents a bijective proof of Theorem \ref{T:4}. We begin by proving Theorem \ref{T:1}, which we then extend to non-rectangular diagrams. 
	
	\begin{theorem}\label{T:1}
		If the diagram $\dg(\lambda)$ is rectangular, then for every $\eta\in\A^+$, there is a bijection $\gamma:\T(\lambda)\to\T(\lambda)$ such that $\gamma(\sigma)\sim\sigma$ and
		\begin{align}\label{eqthm1.1A}
			(\maj,\eta)(\gamma(\sigma))=(\maj,\quinv)(\sigma).
		\end{align}
		The top row is fixed by $\gamma$, and $\des(\sigma\vert_r^{r+1})=\des(\gamma(\sigma)\vert_r^{r+1})$ for all $r\ge 1$.
	\end{theorem}
	The following corollary gives the corresponding statement for the statistics $\eta^*$ on rectangles.
	\begin{corollary}\label{cor:10}
		For any rectangular filling $\sigma$ with entries in the set $[N]$, let $\sigma'$ be obtained by replacing $x$ by $x^c=N+1-x$ and reversing the rows. Then $\gamma(\sigma)'\sim\sigma'$ and
		\begin{align}\label{eqthm1.1A2}
			(\maj,\eta^*)(\gamma(\sigma)')=(\maj,\inv)(\sigma'),
		\end{align}
		where the bottom row of $\sigma'$ is fixed.
	\end{corollary}
	\begin{proof}
		By Remark~\ref{Rem:3}, $\eta^*(\gamma(\sigma)')=\eta(\gamma(\sigma))$; similarly,  $\inv(\sigma')=\quinv(\sigma)$. Since $\des(\sigma\vert_r^{r+1})=\des(\gamma(\sigma)\vert_r^{r+1})$ for every pair of consecutive rows, it follows that $\maj(\sigma')=\maj(\gamma(\sigma)')$. Combining these facts with \eqref{eqthm1.1A} proves \eqref{eqthm1.1A2}.
	\end{proof}
	We now focus on establishing (\ref{eqthm1.1A}). The main step is to construct the bijection $\gamma$ in Theorem~\ref{T:1} for two-row fillings.
	We first list the eight choices of $\quinv$-quadruple sets $\S$ to identify the patterns on which $\eta$ and $\quinv$ differ.
	Proposition~\ref{pro1A} reduces the rectangular case to this two-row construction.
	Lemma~\ref{L:1} treats fillings whose columns are all descents or all non-descents, selecting a suitable top-row order and the flip operators.
	Proposition~\ref{prop:thm9rec} then combines these two cases by showing that the contributions involving one column of each type remain unchanged.
	Finally, we apply the bijection $\gamma$ to each rectangular component of a filling independently. 
	\begin{lemma}\label{L:setSi}
		Let $\S^{*}=\{(z,w,u,v)\mid (z=w>v>u), (u\ge z=w>v),(v>u\ge z=w)\}$ be the set of patterns with $z=w$ in Definition ~\ref{Def:typeqinv} (1), and let
		\[
		B=\{(z,w,u,v)\mid z>v\ge w>u,\ u\ge z>v\ge w\}
		\]
		be the set of patterns in Definition \ref{Def:typeqinv} (4).
		For $z>w$, define the following pairs of patterns:
		\begin{align}\label{E:seta}
			a_1&=(z>w>v>u,\ z>w>u>v),\\
			a_2&=(u>v\ge z>w,\ v>u\ge z>w),\notag\\
			a_3&=(z>u>v\ge w,\ z>v>u\ge w),\notag\\
			a_4&=(v\ge z>w>u,\ u\ge z>w>v).\notag
		\end{align}
		Choose one pattern from each pair $a_i$ for $1\le i\le 4$. We use $\alpha_i=0$ and $\alpha_i=1$ to denote the first and second choices, respectively. In particular, $\alpha_3=\alpha_4$. 
		Together with $B$, their reflected patterns $(w,z,v,u)$, and $\S^*$, all eight $\quinv$-quadruple sets are numbered as follows
		\[
		\begin{array}{c|cccccccc}
			\S_i&\S_1&\S_2&\S_3&\S_4&\S_5&\S_6&\S_7&\S_8\\\hline
			\alpha_1\,\alpha_2\,\alpha_3&000&010&100&110&001&101&011&111
		\end{array}
		\]
		where $\S_2$ is the set $\S$ in Example \ref{Example:S}.
	\end{lemma}
	\begin{proof}
		Condition (1) in Definition \ref{Def:typeqinv} reduces the choice of $\S$ to the case $z>w$. Condition (4) includes $B$, and condition (2) excludes its bottom-swapped patterns. The remaining patterns form the four pairs $a_i$ in \eqref{E:seta}. Condition (2) chooses one member of each pair, while condition (3) couples the last two choices in $a_3$ and $a_4$. Thus there are eight choices.
	\end{proof}
	These yield an indexed family of eight statistics in $\A^+$ and eight dual statistics. We now carry out a 
	reduction needed for the proof of Theorem \ref{T:1}.
	
	\begin{proposition}\label{pro1A}
		If Theorem~\ref{T:1} holds for two-row rectangles, it holds for every rectangle.
	\end{proposition}
	\begin{proof}
		The one-row case is the identity map. For an $m$-row filling $\sigma$, first apply the inductive bijection $\gamma$ to the bottom $m-1$ rows, obtaining $\tau=\gamma(\sigma\vert_{1}^{m-1})$. Then
		\begin{align}\label{eqlem5.1A}
			(\eta,\maj)(\tau)&=(\quinv,\maj)(\sigma|_{1}^{m-1}),
		\end{align}
		where the entries in row $m-1$ of $\sigma$ is unchanged and $\des(\sigma|_{r-1}^r)=\des(\tau|_{r-1}^r)$ for all $1<r\le m-1$.  Let $\tilde{\sigma}$ be an $m$-row filling obtained by adding the top row of $\sigma$ to the top of $\tau$, i.e., 
		\begin{align*}
			\tilde{\sigma}|_{1}^{m-1}=\tau\,\quad \mbox{ and }\quad \tilde{\sigma}|_{m-1}^{m}=\sigma|_{m-1}^{m}.
		\end{align*}
		Apply the two-row bijection $\gamma$ to the top two rows $\sigma|_{m-1}^m$, obtaining $\tilde\tau=\gamma(\sigma|_{m-1}^m)$. Fix a deterministic sequence of adjacent swaps taking the bottom row of $\sigma|_{m-1}^m$ to that of $\tilde\tau$, say, 
		\begin{align*}
			\delta_{j_1}^1\circ \cdots\circ\delta_{j_k}^1(\sigma\vert_{m-1}^{m})=\tilde{\tau}.
		\end{align*} 
	By the properties of $\gamma$ for two-row fillings, 
		\begin{align}\label{eqlem5.2A}
		(\eta,\des)(\tilde{\tau})&=(\quinv,\des)(\sigma|_{m-1}^{m}).
	\end{align}
		Define the map $\gamma:\T(\lambda)\rightarrow \T(\lambda)$ for $m$-row rectangular diagrams $\lambda$ as follows:
			\begin{align}\label{E:gamma_square}
			\gamma(\sigma)=\delta_{j_1}^{m-1}\cdots\delta_{j_k}^{m-1}(\tilde\sigma).
		\end{align}
	The top two rows of $\gamma(\sigma)$ are exactly $\tilde\tau$. Lemma~\ref{L:delta} implies that each of the flip operators in (\ref{E:gamma_square}) preserves the number of descents and $\S$-quadruples between any two consecutive rows below row $m-1$. Therefore, 
	\begin{align}
		\maj(\sigma)=\sum_{r=1}^{m-1}(m-r)\des(\sigma|_{r}^{r+1})&=\sum_{r=1}^{m-1}(m-r)\des(\gamma(\sigma)|_{r}^{r+1})=\maj(\gamma(\sigma)),\notag\\
		\label{eqlem5.4A}
		\eta(\gamma(\sigma)|_{1}^{m-1})-\eta(\gamma(\sigma)|_{m-1}^{m-1})&=\eta(\tau)-\eta(\tau|_{m-1}^{m-1})
	\end{align}
	by (\ref{E:gamma_square}).
	When a tableau contains only one row, both statistics $\quinv$ and $\eta$ coincide with coinversions in Remark \ref{Rem:coinv}. Thus, \eqref{eqlem5.1A}, \eqref{eqlem5.2A} and \eqref{eqlem5.4A} give
	\begin{align*}
		\quinv(\sigma)&=\quinv(\sigma|_{1}^{m-1})-\quinv(\sigma|_{m-1}^{m-1})+\quinv(\sigma|_{m-1}^{m}) \notag\\
		&=\eta(\tau)-\eta(\tau|_{m-1}^{m-1})+\quinv(\sigma|_{m-1}^{m}) \notag\\
		&=\eta(\tau)-\eta(\tau|_{m-1}^{m-1})+\eta(\tilde{\tau}) \notag\\
		&=\eta(\gamma(\sigma)|_{1}^{m-1})-\eta(\gamma(\sigma)|_{m-1}^{m-1})+\eta(\gamma(\sigma)|_{m-1}^{m}) \notag\\
		&=\eta(\gamma(\sigma)).
	\end{align*}
Since the map $\gamma$ as constructed in (\ref{E:gamma_square}) is invertible, it follows that $\gamma$ is a bijection, completing the proof.
	\end{proof}
	
	For a quadruple $\tau=(z,w,u,v)$, let $\tau^r=(w,z,v,u)$ be its {\em \color{blue} reflection} and let 
	$\tau^*=(z,w,v,u)$ be its {\em \color{blue} bottom-swapped pattern}.
	Call $\tau$ \emph{\color{blue} neutral} if $\chi(z>u)=\chi(w>v)$, and \emph{\color{blue} non-neutral} otherwise. It is \emph{\color{blue} decreasing} if $z>w$ and \emph{\color{blue} increasing} if $z<w$.
	
	\begin{lemma}\label{L:1}
		Theorem~\ref{T:1} holds for two-row rectangular fillings in which all columns are descents or all columns are non-descents.
	\end{lemma}
	\begin{proof}
		For any $\eta\in\A^+$, the difference between $\quinv(\sigma)$ and $\eta(\sigma)$ is given by 
		\begin{align}\label{eqthm1.1}
			\quinv(\sigma)&=\eta(\sigma) +\sum_{\tau\in\sigma,\,\Q_{\S}(\tau)=0}
			\Q(z,u,v)-\sum_{\tau\in\sigma,\,\Q(z,u,v)=0}\Q_{\S}(\tau),
		\end{align}
		where the sums are over all quadruples $\tau=(z,w,u,v)$ of $\sigma$ satisfying the indicated conditions.
		Let 
		$a_i^r=\{\tau \mid \tau^r\in a_i \}$ be the reflection set of $a_i$ for $1\le i\le 4$, where $a_i$ is given in (\ref{E:seta}). For each set $(a_i\cup a_i^r)\cap \S$, exactly one selected pattern satisfies $\Q(z,u,v)=0$; denote it by $\tau_i^*$ and its bottom-swapped pattern by $\tau_i$. It follows that $\S\cap (a_i\cup a_i^r)=\{\tau_i^*,(\tau_i^*)^r\}$ and only $\tau_i^*$ contributes $1$ to the second sum of (\ref{eqthm1.1}); namely, $$\Q_{\S}(\tau)=\chi(\tau\text{ has pattern }\tau_i^*)=1.$$ 
		On the other hand, Definition \ref{Def:typeqinv} (2) implies $\tau_i\not \in \S$. Since $\tau_i^*$ satisfies $\CMcal{Q}(z,u,v)=0$ with $u\ne v$, its bottom-swapped pattern $\tau_i$ satisfies $\CMcal{Q}(z,u,v)=1$. Therefore, only pattern $\tau_i$ in the set $\{\tau_i,\tau_i^r\}$ contributes $1$ to the first sum of (\ref{eqthm1.1}); namely, 
		 $$\CMcal{Q}(z,u,v)=\chi(\tau\text{ has pattern }\tau_i)=1.$$
        For every pattern $(z,w,u,v)$ in $\S^*\cup B$, both itself and its reflection $(w,z,v,u)$ have the same indicators for $\quinv$ and $\eta$; that is, $\Q(z,u,v)=\Q_{\S}(z,w,u,v)$, $\Q(w,v,u)=\Q_{\S}(w,z,v,u)$. This implies that the contributions of $\tau$ and $\tau^*$ to the sums of (\ref{eqthm1.1}) are zero for $\tau \in \S^*\cup B$. It follows that (\ref{eqthm1.1}) becomes
		\begin{align}\label{eqthm1.12}
			\quinv(\sigma)=
			\eta(\sigma)+\sum_{\tau\in\sigma}\sum_{1\le i\le 4}
			\chi(\tau\text{ has pattern }\tau_i)-\chi(\tau\text{ has pattern }\tau_i^*).
		\end{align}
	    See Example \ref{Example:diff}.
		If all columns of $\sigma$ are descents, only $i=1$ contributes to the right-hand side of (\ref{eqthm1.12}). 
		In addition, if $\tau_1$ is increasing (or decreasing), then the right-hand side of (\ref{eqthm1.12}) vanishes for all two-row fillings $\sigma$ with exclusively descent columns and weakly decreasing (or increasing) top-row entries. This implies
		\begin{align*}
			\quinv(\sigma)=\eta(\sigma)
		\end{align*} 
		Given an arbitrary two-row filling $\tau$ with all descent columns, apply queue-inversion operators to sort its top-row entries, obtaining a filling $\sigma$ whose top row is in the chosen order, and write $	\tau=\rho_{i_1}\cdots\rho_{i_k}(\sigma)$. Define
		\begin{align*}
			\gamma(\tau)=\delta_{i_1}\cdots\delta_{i_k}(\sigma)
		\end{align*}
		where $\gamma$ fixes the top row of $\tau$. Lemma~\ref{L:top} asserts that 
		 \begin{align*}
			(\quinv,\maj)(\rho_i(\sigma))=(\eta,\maj)(\delta_i(\sigma)),
		\end{align*}
		which yields $(\quinv,\maj)(\tau)=(\eta,\maj)(\gamma(\tau))$. The other case follows by the same strategy, namely that $\maj(\sigma)=\des(\sigma)=0$; thus we omit the details.
	\end{proof}
	
	\begin{example}\label{Example:diff}
		For the set $\S$ in Example \ref{Example:S}, we have
		\begin{align*}
			\quinv(\sigma)&= \eta(\sigma)
			+ \sum_{(z,w,u,v)\in \sigma}\bigg( \chi(v>u\ge w>z)+ \chi(z>v>u\ge w)\notag \\
			&\qquad+\chi(w>z>v>u) +\chi(u\ge z>w>v) -\chi(v\ge z>w>u) \notag\\
			&\qquad-\chi(z>u>v\ge w)- \chi(u>v\ge w>z)-\chi(w>z>u>v)  \bigg).
		\end{align*}
	\end{example}

	\begin{proposition}\label{prop:thm9rec}
		Theorem~\ref{T:1} holds for all rectangular diagrams.
	\end{proposition}
	\begin{proof}
	By Proposition~\ref{pro1A} and Lemma \ref{L:1}, it remains to treat a two-row filling $\sigma$ with both descent and non-descent columns. 
	
	Fix $i\in \{3,4\}$. Since $\tau_i^*$ is an element of $(a_i\,\cup\, a_i^r)\cap \S$ in (\ref{eqthm1.12}) and it is a non-neutral quadruple with $\Q(z,u,v)=0$, by Definition \ref{Def:typeqinv} (3), both $\tau_i$ and $\tau_i^*$ are uniformly decreasing or increasing.
	Without loss of generality, we assume that $\tau_i,\tau_i^*$ are increasing quadruples for $i\in\{3,4\}$. For all tableaux $\sigma\in\T(\lambda)$ whose top-row entries are  weakly decreasing, it follows that
	\begin{align}\label{E:chi34}
		\chi(\pi \mbox{ has pattern }\tau_i)=\chi(\pi \mbox{ has pattern }\tau_i^*)=0
	\end{align}
	for $i\in \{3,4\}$ and any decreasing $\pi \in\sigma$ in \eqref{eqthm1.12}. Let $\sigma_{>}$ and $\sigma_{\le}$ denote the tableaux consisting of the descent columns and non-descent columns of $\sigma$, respectively. Applying the bijection $\gamma$ in Lemma \ref{L:1} on $\sigma_{>}$ and $\sigma_{\le}$ gives two tableaux $\tau_{>}$ and $\tau_{\le}$ satisfying $\eta(\tau_{>})=\quinv(\sigma_{>})$ and $\eta(\tau_{\le})=\quinv(\sigma_{\le})$.
	
	Define $\gamma(\sigma)=\tau$ as the tableau formed by placing $\tau_{>}$ and $\tau_{\le}$ according to the positions of descent and non-descent columns of $\sigma$. In particular, $\gamma(\sigma)_{>}=\tau_{>}$, $\gamma(\sigma)_{\le}=\tau_{\le}$ and the top rows of $\sigma$ and $\tau$ are identical.
	It is immediate that $\gamma(\sigma)\sim \sigma$ and  $\maj(\sigma)=\maj(\tau)$. We only need to prove $\quinv(\sigma)=\eta(\tau)$. 
	
	Let $\quinv(\tau_{>},\tau_{\le})$ and $\eta(\tau_{>},\tau_{\le})$ count the number of queue inversion triples, and the number of $\S$-quadruples, formed by one descent column and one non-descent column of $\tau$, respectively. Since 
		\begin{align*}
		\quinv(\sigma)&=\quinv(\sigma_{>},\sigma_{\le})+\quinv(\sigma_{>})+\quinv(\sigma_{\le}), \\
		\eta(\tau)&=\eta(\tau_{>},\tau_{\le})+\eta(\tau_{>})+\eta(\tau_{\le}),
	\end{align*}
	it is equivalent to proving $\quinv(\sigma_{>},\sigma_{\le})=\eta(\tau_{>},\tau_{\le})$. We claim that 
		\begin{align}\label{E:quin=eta}
		\quinv(\tau_{>},\tau_{\le})=\eta(\tau_{>},\tau_{\le}).
	\end{align}
    Returning to \eqref{eqthm1.12}, we have, for $i\in \{1,2\}$, that $\tau_i$ and $\tau_i^*$ are neutral. By  (\ref{E:chi34}), the cases $i\in \{1,2\}$ contribute nothing to the difference $\quinv(\tau_{>},\tau_{\le})-\eta(\tau_{>},\tau_{\le})$. Together with (\ref{E:chi34}), we arrive at (\ref{E:quin=eta}). Consequently, the proof reduces to establishing 
	\begin{align}\label{eqpro1A}
		\quinv(\tau_{>},\tau_{\le})=\quinv(\sigma_{>},\sigma_{\le}).
	\end{align}
    That is, we next show that $\quinv(\sigma_{>},\sigma_{\le})$ is unchanged by the separate rearrangements within descent columns and within non-descent columns. 
    
    Let $D$ and $N$ be the sets of descent and non-descent positions in $\sigma$, respectively. Write columns in $\sigma$ as
    $(z_d,u_d)$ for $d\in D$ and $(z_{n},v_{n})$ for $n\in N$, so that $z_d>u_d$ and $z_n\leq v_n$.
    Let $c$ be the number of queue inversion triples involving an adjoined zero and top-row
    entries from one descent and one non-descent columns. For $d\in D$ and $n\in N$, let
    \[
    q_{d,n}(\sigma)=
    \begin{cases}
    	\Q(z_d,u_d,v_n), & d<n,\\
    	\Q(z_n,v_n,u_d), & n<d,
    \end{cases}
    \]
    be the number of remaining triples located in columns $d$ and $n$ between the two rows of $\sigma$. If $d<n$, then $\Q(z_d,u_d,v_n)=\chi(u_d<v_n<z_d)=\chi(u_d<v_n)-\chi(z_d\leq v_n)$. Otherwise $d>n$ and 
    $\Q(z_n,v_n,u_d)=1$. Hence 
    \begin{align*}
    \quinv(\sigma_{>},\sigma_{\le})
    &=c+
    \sum_{(d,n)\in D\times N}q_{d,n}(\sigma)\\
    &=c
    +\sum_{(d,n)\in D\times N}
    (\chi(u_d<v_n)-\chi(z_d\leq v_n))+\#\{(d,n)\in D\times N:n<d\}.
    \end{align*}
    Since the right-hand side depends only on the top row, the position sets $D,N$ and the multisets 
    $\{u_d:d\in D\}$ and $\{v_n:n\in N\}$, it is invariant under $\gamma$, thus implying (\ref{eqpro1A}).
    Finally, for arbitrary top-row entries, we choose a fixed stable sorting word as in Lemma~\ref{L:1}. That is, if  $\pi=\rho_{i_1}\cdots\rho_{i_k}(\sigma)$ where $\sigma$ is a two-row rectangular filling with weakly decreasing top-row entries, we define 
    \begin{align}\label{E:gamma_final}
    	\gamma(\pi)=\delta_{i_1}\cdots\delta_{i_k}(\gamma(\sigma)).
    \end{align}
    Lemma \ref{L:top} states that $\quinv(\sigma)-\quinv(\pi)=\eta(\gamma(\sigma))-\eta(\gamma(\pi))$ and $\maj(\pi)=\maj(\gamma(\pi))$. Finally, since $\quinv(\sigma)=\eta(\gamma(\sigma))$, we obtain $\quinv(\pi)=\eta(\gamma(\pi))$, as desired.
	\end{proof}
	
	\begin{example}
		For $\S=\S_2$, consider the following two-row filling:
		\begin{align*}
			\centering
			\tau=\begin{ytableau}[] 2&3 &2& 4& 4& 5& 4& 5& 6& 8& 7\\ 3& 5& 1& 3& 7& 6& 5& 1& 4& 2& 6\end{ytableau}\,.
		\end{align*}
		Note that  $\S\cap a_3=\{(z,w,u,v):z>u>v\ge w\}$ and $\S\cap a_3^r=\{(z,w,u,v):w>v>u\ge z\}$. Since only one of these two quadruples satisfies $\CMcal{Q}(z,u,v)=0$, we conclude that $\tau_3^*=(z>u>v\ge w)$ and it is decreasing. By Proposition \ref{prop:thm9rec}, we first sort the top-row entries of $\tau$ into a  weakly increasing sequence. That is, let
		\begin{align*}
			\centering
			\sigma=\rho_2\circ\rho_6\circ\rho_{10}(\tau)
			=\begin{ytableau}[] 2&2 &*(green)3& *(green)4& 4& 4& 5& *(green)5& *(green)6& *(green)7& *(green)8\\ 3& 5& *(green)1& *(green)3& 7& 6& 5& *(green)1& *(green)4& *(green)2& *(green)6\end{ytableau}
		\end{align*}
		The descent and non-descent positions are $D=\{3,4,8,9,10,11\}$ and $N=\{1,2,5,6,7\}$, and
		\begin{align*}
			\sigma_{>}=\begin{ytableau}[] 3&4&5 & 6& 7 & 8\\ 1& 3& 1& 4& 2 & 6
			\end{ytableau} \, ,\quad\quad
			\sigma_{\le}=\begin{ytableau}[] 2&2&4 & 4& 5\\ 3& 5& 7& 6& 5
			\end{ytableau}\, .
		\end{align*}
	 We apply Lemma \ref{L:1} to find $\gamma(\sigma_{>})$ and  $\gamma(\sigma_{\le})$. Note that $\S\cap a_1=\{(z,w,u,v):z>w>v>u\}$ and $\S\cap a_1^r=\{(z,w,u,v):w>z>u>v\}$, and only the second one satisfies $\CMcal{Q}(z,u,v)=0$. Therefore, we obtain that $\tau_1^*=(w>z>u>v)$ and it is increasing. Similarly, $\tau^*_2=(u>v\ge w>z)$. Hence we permute the top-row entries of $\sigma_{>}$ and $\sigma_{\le}$ into weakly decreasing sequences by $\rho_i$'s, respectively. That is, 
		\begin{align*}
			&\begin{ytableau}[] 3&4&5 & 6& 7 & 8\\ 1& 3& 1& 4& 2 & 6
			\end{ytableau} \xrightarrow{\rho_5\circ \rho_4\circ \rho_3 \circ\rho_2\circ \rho_1}
			\begin{ytableau}[] 4&5 & 6& 7 & 8 & 3\\ 3& 1& 4& 1& 6 & 2
			\end{ytableau} \xrightarrow{\rho_4\circ \rho_3 \circ\rho_2\circ \rho_1}
			\begin{ytableau}[] 5 & 6& 7 & 8 & 4& 3\\ 3& 4& 1& 6& 1 & 2
			\end{ytableau}\\
			&\xrightarrow{ \rho_3 \circ\rho_2\circ \rho_1}\begin{ytableau}[] 6& 7 & 8 & 5 &4& 3\\ 3& 4& 6& 1& 1 & 2\end{ytableau}
			\xrightarrow{ \rho_2\circ \rho_1}\begin{ytableau}[] 7 & 8& 6 & 5 &4& 3\\ 3& 6& 4& 1& 1 & 2\end{ytableau}
			\xrightarrow{ \rho_1}\begin{ytableau}[] 8 & 7& 6 & 5 &4& 3\\ 3& 6& 4& 1& 1 & 2\end{ytableau},
		\end{align*}
		and
		\begin{align*}
			&\begin{ytableau}[] 2&2& 4 & 4& 5\\ 3& 5& 7& 6& 5
			\end{ytableau} \xrightarrow{\rho_4\circ \rho_3 \circ\rho_2 }
			\begin{ytableau}[] 2& 4 & 4& 5 & 2\\ 3& 5& 7& 6& 5
			\end{ytableau} \xrightarrow{ \rho_3 \circ\rho_2\circ \rho_1}
			\begin{ytableau}[] 4 & 4& 5 & 2 & 2\\ 5& 7& 6& 3& 5
			\end{ytableau}\\
			&\xrightarrow{ \rho_2}\begin{ytableau}[] 4& 5 & 4 & 2 & 2\\ 5& 7& 6& 3& 5 \end{ytableau}
			\xrightarrow{  \rho_1}\begin{ytableau}[] 5 & 4& 4 & 2 & 2\\ 5& 7& 6& 3& 5\end{ytableau}.
		\end{align*}
		
		We then recover the original top row of each subtableau using the corresponding $\delta_i$
		operators, as in Lemma~\ref{L:1}. This yields
		\[
		\gamma(\sigma_{>})=
		\begin{ytableau}
			3&4&5&6&7&8\\
			2&1&1&4&6&3
		\end{ytableau},
		\qquad
		\gamma(\sigma_{\le})=
		\begin{ytableau}
			2&2&4&4&5\\
			3&5&7&6&5
		\end{ytableau}.
		\]
		Place these columns back in the positions $D$ and $N$ to obtain
		\[\gamma(\sigma)=\begin{ytableau}[] 2&2 &*(green)3& *(green)4& 4& 4& 5& *(green)5& *(green)6& *(green)7& *(green)8\\ 3&5& *(green)2& *(green)1& 7& 6& 5& *(green)1& *(green)4& *(green)6& *(green)3\end{ytableau}.\]
		Finally, recover the top row of $\tau$:
		\[
		\gamma(\tau)
		=
		\delta_2\,\delta_6\,\delta_{10}(\gamma(\sigma))
		=
		\begin{ytableau}
			2&3&2&4&4&5&4&5&6&8&7\\
			3&2&5&1&7&5&6&1&4&3&6
		\end{ytableau}\,.
		\]
		Direct calculation gives $\quinv(\tau)=\eta(\gamma(\tau))=72$ and $\maj(\tau)=\maj(\gamma(\tau))=6$.
	\end{example}

	We now extend the bijection to arbitrary shapes.
	
	\begin{proof}[Proof of Theorem~\ref{T:4}]
		For the decomposition \eqref{E:sigmadec}, define
		\[
		\gamma(\sigma)=\gamma(\sigma_n)\sqcup\cdots\sqcup\gamma(\sigma_1).
		\]
		This is a bijection fixing the row multisets and the number of descents between any two consecutive rows in each component. It therefore preserves the major index and, by Lemma~\ref{L:crossblock}, all unequal-height queue inversion contributions. Theorem~\ref{T:1} handles the contributions within each rectangle, so
		\[
		\eta(\gamma(\sigma))=\sum_j\eta(\gamma(\sigma_j))
		+\sum_{(a,b,c)\in\sigma}\Q_{\S}(a,b,c)=\quinv(\sigma).
		\]
		In particular, if $\langle\sigma\rangle$ consists of fillings row-equivalent to $\sigma$ within each rectangle, then
		\begin{align}\label{E:eq162}
			\sum_{\tau\in\langle\sigma\rangle}t^{\maj(\tau)}q^{\quinv(\tau)}
			=\sum_{\tau\in\langle\sigma\rangle}t^{\maj(\tau)}q^{\eta(\tau)}.
		\end{align}
		These component classes partition $[\sigma]$, proving \eqref{E:eq16} for $\eta\in\A^+$.
		
		Furthermore, using the bijection $\sigma\mapsto\sigma'$ from Corollary~\ref{cor:10}, we define $\gamma(\sigma)=(\gamma(\sigma'))'$ and obtain (\ref{eqthm1.1A2}) for non-rectangular diagrams and hence, by (\ref{E:T1}), gives (\ref{E:eq16}) for the dual statistics $\eta^*\in \A\,\backslash \,\A^+$.

	\end{proof}
	
	\section{Canonical and dual canonical tableaux}\label{S:can}
	
	The main goal of this section is to introduce the notions of canonical and dual canonical tableaux, which are central to the compact formulas for the modified Macdonald polynomials in Theorem \ref{T:2}.
	
	\begin{definition}[{\em \color{blue}(dual) Canonical tableaux}]\label{Def:3}
		Write $\sigma=\sigma_n\sqcup\cdots\sqcup\sigma_1$, where $\sigma_j$ is the rectangular component of height $j$. The filling $\sigma$ is \emph{\color{blue}canonical} (respectively, \emph{\color{blue}dual canonical}) if every nonempty component $\sigma_j$ is canonical (respectively, dual canonical).
		
		For a rectangle of height $n$, set $\sigma(n+1,i)=0$. A filling is {\em \color{blue} canonical} if, at every pair of consecutive rows, the following conditions hold.
		\begin{enumerate}
			\item[(I)] If two columns $(a,b)$ and $(a,c)$ occur in this order, then
			\[
			a>b\ge c,\qquad c\ge a>b,\qquad\text{or}\qquad b\ge c\ge a.
			\]
			Equivalently, below each fixed upper entry $a$, the entries less than $a$ occur first in weakly decreasing order, followed by the entries at least $a$ in weakly decreasing order.
			\item[(II)] Sort the non-descent columns by their upper entries in weakly decreasing order. The corresponding lower entries must also be weakly decreasing.
		\end{enumerate}
		A filling is {\em \color{blue} dual canonical} if it satisfies (I) and the version of (II) with descent columns in place of non-descent columns. 
	\end{definition}
	
	\begin{example}\label{Eg:5}
		The following fillings $\sigma$ and $\bar\sigma$ are canonical and dual canonical, respectively:
		\begin{align*}
			\sigma&=\begin{ytableau}9&9&8&8&8&5&5&5&5&3&3\\4&3&6&3&9&4&3&2&9&2&6\\2&2&8&7&8&1&6&4&6&2&7\end{ytableau}\,,\\[2pt]
			\bar\sigma&=\begin{ytableau}9&9&8&8&8&5&5&5&5&3&3\\6&4&4&3&9&3&3&2&9&2&6\\8&2&2&1&8&7&6&4&6&2&7\end{ytableau}\,.
		\end{align*}
	\end{example}
	
	To obtain the compact formulas, we group fillings into families $\CMcal{G}(\sigma)$ indexed by canonical representatives, see Definition~\ref{D:family}.
	We shall introduce blocks to specify which entries are rearranged together and to track the resulting changes in $\eta$.
	Our blocks differ from those in \cite{CHO22} by separating descent and non-descent columns in two consecutive rows.
	The operators $\delta_i^r$ extend these row rearrangements to the whole filling, producing the family $\CMcal{G}(\sigma)$ from $\sigma$.
	\begin{definition}[{\em \color{blue}Blocks}]\label{Def:block}
		Let $p$ be a subsequence of a row, taken in left-to-right order, and let $a$ be the sequence immediately above it. A \emph{\color{blue}descent block} is the maximal subsequence with $a_i>p_i$; a \emph{\color{blue}non-descent block} is the maximal subsequence with $a_i\le p_i$. A \emph{\color{blue}neutral block} consists of all entries with a fixed upper neighbor. Blocks need not be contiguous. Each row has at most one descent block and one non-descent block, and its neutral blocks form a partition of its entries. We use zero upper neighbors above the top row.
	\end{definition}
	
	For example, the bottom row of $\sigma$ in Example \ref{Eg:5} has neutral block $(2,7,6)$ below the upper entry $3$, and non-descent block $(8,7,6,4,2,7)$.
	For a word $p=(p_1,\ldots,p_k)$ and $\omega\in S_k$, write
	\begin{equation}\label{E:piep}
		\omega(p)=(p_{\omega^{-1}(1)},\ldots,p_{\omega^{-1}(k)}).
	\end{equation}

	
	\begin{definition}[{\em \color{blue}Sorting permutations}]\label{Def:perm2}
		Consider the two-row filling $\sigma$ in Figure \ref{Fig:1}. 
		
		Let $\pi^{\bullet}$ be the permutation of $[k]$ that sorts all entries of $a=(a_1,\ldots,a_k)$ into weakly decreasing sequence $\pi^{\bullet}(a)$ such that if $a_i=a_j$ with $i<j$, then $\pi^{\bullet}(i)<\pi^{\bullet}(j)$.
		
		Let $\pi^{\diamond}$ be the permutation of $[k]$ such that $\pi^{\diamond}(a)$ starts with the entries of $a$ in descent columns of $\sigma$ in the left-to-right order, followed by the entries of $a$ in non-descent columns in weakly decreasing order such that if $a_i=a_j\le \min(p_i,p_j)$ with $i<j$, then $\pi^{\diamond}(i)<\pi^{\diamond}(j)$.
		
		\begin{figure}[ht]
			\centering
			\includegraphics[scale=0.4]{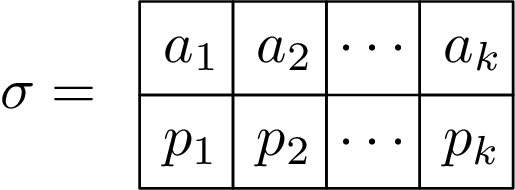}
			\caption{A two-row rectangular tableau}
			\label{Fig:1}
		\end{figure}
	\end{definition}
	
	\begin{example}
		For the second row $a=(4,3,6,3,9,4,3,2,9,2,6)$ of $\sigma$ in Example~\ref{Eg:5}, the sorting permutations, in one-line notation, give
		\[
		\begin{aligned}
			\pi^\diamond&=1268349\,10\,5\,11\,7,
			&\pi^\diamond(a)&=(4,3,9,4,9,6,6,3,3,2,2),\\
			\pi^\bullet&=5738169\,10\,2\,11\,4,
			&\pi^\bullet(a)&=(9,9,6,6,4,4,3,3,3,2,2).
		\end{aligned}
		\]
		
		
	\end{example}

\begin{lemma}\label{L:equ2Def}
	Let $\sigma$ be the two-row rectangular tableau in Figure \ref{Fig:1}. Suppose that $p_i\ge a_i$ for every $i$. 
	Then $\sigma$ is canonical if and only if $\pi^{\diamond}(p)$
	weakly decreases.
\end{lemma}
\begin{proof}
	Definition \ref{Def:3} (II) holds if and only if $p_i\ge p_j$ whenever $a_i>a_j$, or whenever $a_i=a_j$ with $i<j$, that is, for all $1\le \sort^{\diamond}(i)<\sort^{\diamond}(j)\le k$. This is equivalent to saying that $\sort^{\diamond}(p)$ is weakly decreasing.
\end{proof}

	We recall standard facts on multiset permutations, following \cite{CHO22}.
	\begin{definition}[{\em \color{blue} Multiset permutations and length}]\label{Def:reduced}
		Let $\mathrm{Sym}(p)$ be the set of distinct rearrangements of sequence $p$. For $w\in\mathrm{Sym}(p)$, let $\tilde w$ be the unique shortest permutation with $\tilde w(p)=w$. Let $$\widetilde{\mathrm{Sym}}(p)=\{\tilde{w}\,\vert\, w\in \mathrm{Sym}(p)\}$$ 
		and set $L_p(w)=\ell(\tilde w)$
		where $\ell(\tilde{w})$ is the minimum number of simple transpositions $s_i=(i\,i+1)$ in an expression for the permutation $w$. A shortest expression is called {\em \color{blue} reduced}. In particular, $L_p(w)=\inv(w)$ when $p$ is weakly increasing.
	\end{definition}

	\begin{remark}\label{Rem:5}
		If $p$ is weakly decreasing, then $L_p(w)=\coinv(w)$ as in Remark \ref{Rem:coinv}. Indeed, reversing both words preserves adjacent-swap distance, so $L_p(w)=L_{p^r}(w^r)=\inv(w^r)=\coinv(w)$.
	\end{remark}
	
	\begin{definition}[{\em \color{blue}Lifted block rearrangements}]\label{Def:delta}
		For a reduced expression $\tilde w=s_{i_\ell}\cdots s_{i_1}$, put
		\[
		\Delta_{\tilde w}^{r}=\delta_{i_\ell}^{r}\cdots\delta_{i_1}^{r}.
		\]
		Lemma \ref{L:braid}, together with the commutation for disjoint adjacent pairs, i.e., $\delta_i^r\delta_j^r=\delta_j^r\delta_i^r$ for $|i-j|\ge 2$, makes this independent of the reduced expression. When row $r$ is the top row, we drop $r$ from the notation as in Definition \ref{Def:2} and write 
		\begin{align*}
			\Delta_{\tilde w}=\delta_{i_\ell}\cdots\delta_{i_1}.
		\end{align*} 
		To rearrange a block in row $r$, first use $\pi^\bullet$ to sort a neutral block, and $\pi^\diamond$ for a non-descent block. Denote the chosen permutation by $\pi^\star$. For a canonical filling,
		all two permutations preserve the relative order of columns with equal entries in row $r+1$. With respect to this new block,
		let $\tilde w$ be the shortest permutation that gives the desired
		rearrangement, fixing all positions outside the block. Define
		\begin{equation}\label{E:delta2}
			\delta_{\tilde w}^{r\,\star}(\tau)
			=
			\Delta_{(\pi^\star)^{-1}}^{r+1}
			\circ \Delta_{\tilde w}^{r}
			\circ \Delta_{\pi^\star}^{r+1}(\tau).
		\end{equation}
		The first and last operators act on rows $1$ to $r+1$,
		where the first restores the order of row $r+1$.
		For the top row, these two steps are unnecessary and
		$\delta_{\tilde w}^{r\,\star}=\Delta_{\tilde w}^{r}$.


		The two stages are performed in the following order: first rearrange the non-descent blocks, then rearrange the neutral blocks. With $v$ denoting the product of the chosen non-descent-block permutations $\pi^{\diamond}$ and $u$ the permutation $\pi^{\bullet}$, each extended by the identity outside its block, write $\tilde{w}=\tilde{u}\tilde{v}$ and
		\begin{equation}\label{E:delta}
			\delta_{\tilde{w}}^r=\delta_{\tilde u\tilde v}^r
			=\delta_{\tilde u}^{r\,\bullet}\,\delta_{\tilde v}^{r\,\diamond}.
		\end{equation}
	\end{definition}

	\begin{example}
		Let $\sigma$ be the canonical filling in Example~\ref{Eg:5}.
		Its bottom row and the row immediately above it are
		$p=(2,2,8,7,8,1,6,4,6,2,7), a=(4,3,6,3,9,4,3,2,9,2,6)$, respectively.
		We illustrate the neutral-block rearrangement that replaces $p$ by $w=(1,7,8,6,6,2,2,4,8,2,7)$. The neutral blocks with upper entries $9$, $4$, and $3$ change as follows:
		\[
		\begin{aligned}
			(8,6)&\longmapsto(6,8),\\
			(2,1)&\longmapsto(1,2),\\
			(2,7,6)&\longmapsto(7,6,2).
		\end{aligned}
		\]
		The blocks with upper entries $6$ and $2$ remain unchanged.
		Now the non-descent block of $w$ is $(7,8,6,4,2,7)$ satisfies $\pi^{\diamond}(7,8,6,4,2,7)=(8,7,7,6,4,2)$, 
		no non-descent-block rearrangement is needed, so we take $\tilde v=v=\mathrm{id}$.
		
		For $r=1$, we use $\pi^{\bullet}=5738169\,10\,2\,11\,4$ to sort the middle-row entries (sequence $a$) of $\sigma$ into weakly decreasing order, and hence
		\[
		\tau_1
		=
		\Delta_{\pi^\bullet}^{2}(\sigma|_1^2)
		=
		\begin{ytableau}
			9&9&6&6&4&4&3&3&3&2&2\\
			8&6&8&7&2&1&2&7&6&4&2
		\end{ytableau}
		\]
		whose top-row and bottom-row entries are $\pi^{\bullet}(a)$ and $\pi^{\bullet}(p)$, respectively. Now the shortest permutation sending
		$\pi^\bullet(p)$ to $\pi^\bullet(w)$ is $\tilde u=s_1s_5s_8s_7=(1\,2)(5\,6)(7\,9\,8)$.
		These transpositions act within the three neutral blocks described above. Applying $\Delta_{\tilde u}^{1}$ to $\tau_1$ gives
		\[
		\tau_2
		=
		\Delta_{\tilde u}^{1}(\tau_1)
		=
		\begin{ytableau}
			9&9&6&6&4&4&3&3&3&2&2\\
			6&8&8&7&1&2&7&6&2&4&2
		\end{ytableau}\,.
		\]
		Finally, applying $\Delta_{(\pi^\bullet)^{-1}}^{2}$ to $\tau_2$ restores the second row to $a$ and changes the bottom row to $w$.
		The top row of $\sigma$ is unaffected. Therefore,
		\[
		\delta_{\tilde u}^{1\,\bullet}(\sigma)=\Delta_{(\pi^\bullet)^{-1}}^{2}(\tau_2)
		=
		\begin{ytableau}
			9&9&8&8&8&5&5&5&5&3&3\\
			4&3&6&3&9&4&3&2&9&2&6\\
			1&7&8&6&6&2&2&4&8&2&7
		\end{ytableau}\,.
		\]
	\end{example}
	
	For a neutral block $P$ with upper entry $a$, define $\alpha(P)$ by replacing every entry at least $a$ by $a$. A shortest rearrangement of $\alpha(P)$ lifts uniquely to a rearrangement of $P$ that preserves the relative order of its non-descent entries. For example, $P=(6,3,10,9)$ below $a=8$ gives $\alpha(P)=(6,3,8,8)$.
	
	\begin{definition}({\em \color{blue}Family $\CMcal{G}(\sigma)$})\label{D:family} For a partition $\lambda=(k^n)$ and a canonical tableau $\sigma\in\T(\lambda)$ with $\sigma(n+1,j)=0$ for $1\le j\le k$, let $\CMcal{G}(\sigma;0)=\{\sigma\}$, and we recursively define $\CMcal{G}(\sigma;r)$ from $\CMcal{G}(\sigma;r-1)$ for $1\le r\le n$ as follows. 
		
		For any $\tau\in \CMcal{G}(\sigma;r-1)$, let $p=(p_1,\ldots,p_k)=\tau\vert_{r}^r$ with upper neighbor $a=(a_1,\ldots,a_k)$. Let $P_1,\ldots,P_{m}$ be the disjoint neutral blocks in row $r$ of  $\Delta_{\pi^{\bullet}}(\tau\vert_{1}^{r+1})$, and define
		\begin{align*}
			\widetilde{\sym}(P)&=\widetilde{\sym}(\alpha(P_1))\times\cdots\times \widetilde{\sym}(\alpha(P_{m}))\\
			&=\{\tilde{u}: u\in \sym(\alpha(P_1))\times \cdots \times \sym(\alpha(P_m))\}.
		\end{align*}
		Let $Q$ be the unique non-descent block in row $r$ of $\Delta_{\pi^{\diamond}}(\tau\vert_{1}^{r+1})$. Let $\sym^*(Q)$ consist of the unique shortest permutations realizing the distinct rearrangements of $Q$ for which every column remains a non-descent, with all positions outside the block fixed.
		
		Define $\CMcal{G}(\sigma;r)=\{\delta_{\tilde{w}}^r(\tau): \tau\in \CMcal{G}(\sigma;r-1), \tilde{w}=\tilde{u}\tilde{v}\in \widetilde{\sym}(P)\times \sym^*(Q)\}$. Finally let  $\CMcal{G}(\sigma)=\CMcal{G}(\sigma;n)$. For a general shape, construct the families independently in the rectangular components.
        Let $\sigma=\sigma_n \sqcup \cdots \sqcup \sigma_1$, where $\sigma_j$ is a canonical rectangular tableau of height $j$. Define
		\begin{equation}\label{E:nonrec}
			\CMcal{G}(\sigma)=
			\{\tau_n\sqcup\cdots\sqcup\tau_1:\tau_j\in\CMcal{G}(\sigma_j)\}.
		\end{equation}
	   
	   Let $\pi^{\dagger}$ be the permutation of $[k]$ such that $\pi^{\dagger}(a)$ starts with entries of $a$ in descent columns of $\sigma$ in weakly decreasing order, such that if $a_i=a_j> \max(p_i,p_j)$ with $i<j$, then $\pi^{\dagger}(i)<\pi^{\dagger}(j)$, followed by the entries of $a$ in non-descent columns in left-to-right order. 
	  
	   For a dual canonical tableau $\bar{\sigma}\in \T(\lambda)$, replace the non-descent block by the descent block and use $\pi^\dagger$. Replace $\alpha(P)$ by $\bar\alpha(P)$, obtained by replacing each entry less than $a$ by $a-1$, and preserve the relative order of the descent entries.
	\end{definition}
	
	\section{Proof of Theorem \ref{T:2}}\label{S:thm3}
	
	In this section, we establish the compact formulas of the modified Macdonald polynomials in Theorem \ref{T:2}, which is a consequence of the theorem below.
	
	\begin{theorem}\label{P:can11}
		Let $\CMcal{C}_\diamond(\lambda)$ and $\CMcal{C}_\dagger(\lambda)$ be the sets of canonical and dual canonical fillings of a given shape $\dg(\lambda)$. For either $\varepsilon\in\{\diamond,\dagger\}$,
		\begin{equation}\label{E:g11}
			\T(\lambda)=\mathop{\dot\bigcup}_{\sigma\in\CMcal{C}_\varepsilon(\lambda)}\CMcal{G}(\sigma).
		\end{equation}
		If $\eta\in\A^+$ satisfies $(v>u\ge z>w)\in\S\cap a_2$, then for $\sigma\in\CMcal{C}_\diamond(\lambda)$,
		\begin{align}
			\sum_{\tau\in\CMcal{G}(\sigma)}t^{\maj(\tau)}q^{\eta(\tau)}
			&=t^{\maj(\sigma)}q^{\eta(\sigma)}d_\diamond(\sigma),\label{E:count11}\\
			\sum_{\tau\in\CMcal{G}(\sigma)}t^{\maj(\tau')}q^{\eta^*(\tau')}
			&=t^{\maj(\sigma')}q^{\eta^*(\sigma')}d_\diamond(\sigma),\label{E:count12}
		\end{align}
	  where $\sigma\mapsto\sigma'$ is the map in Corollary \ref{cor:10}. 
	  
	  If $\eta\in\A^+$ satisfies $(z>w>v>u)\in\S\cap a_1$, then for $\sigma\in\CMcal{C}_\dagger(\lambda)$,
		\begin{align}
			\sum_{\tau\in\CMcal{G}(\sigma)}t^{\maj(\tau)}q^{\eta(\tau)}
			&=t^{\maj(\sigma)}q^{\eta(\sigma)}d_\dagger(\sigma),\label{E:count13}\\
			\sum_{\tau\in\CMcal{G}(\sigma)}t^{\maj(\tau')}q^{\eta^*(\tau')}
			&=t^{\maj(\sigma')}q^{\eta^*(\sigma')}d_\dagger(\sigma).\label{E:count14}
		\end{align}
		The coefficients $d_\varepsilon(\sigma)$ are the products of Gaussian binomial coefficients defined in \eqref{E:coeffd}--\eqref{E:coeffd2} and \eqref{E:coeffdag}--\eqref{E:coeffdag2} for $\sigma=\sigma_n \sqcup \cdots \sqcup \sigma_1$ in (\ref{E:sigmadec}).
		
	\end{theorem}
	Following \cite{CHO22}, we first reduce the factorization to rectangular shapes. 
	
	\begin{lemma}\label{L:rec}
		To prove \eqref{E:count11}--\eqref{E:count14} for arbitrary
		shapes, it suffices to establish \eqref{E:count11} and
		\eqref{E:count13} for rectangular shapes.
	\end{lemma}
	
	\begin{proof}
		Fix $\varepsilon\in\{\diamond,\dagger\}$ and
		$\sigma\in\CMcal{C}_\varepsilon(\lambda)$, and write
		$\sigma=\sigma_n\sqcup\cdots\sqcup\sigma_1$.
		Every rearrangement in the construction of $\CMcal{G}(\sigma)$
		preserves the multiset of entries in each row of each rectangular component. This also preserves the number of descents between each vertical pair
		of consecutive rows. 
		Lemma~\ref{L:delta} preserves the descent counts for the other pairs
		within rows $1,\ldots,r+1$. Restoring row $r+1$ also recovers
		its descent count with row $r+2$.
		
		These properties also hold after applying the map $\tau\mapsto\tau'$ in Corollary \ref{cor:10}.
		Thus, for every $\tau\in\CMcal{G}(\sigma)$,
		\[
		\maj(\tau)=\maj(\sigma),
		\qquad
		\maj(\tau')=\maj(\sigma').
		\]
		Let $c(\tau)$ and $c^*(\tau)$ be the numbers of queue inversion triples and inversion triples involving
		columns from two distinct rectangular components of $\tau$, respectively. It follows from Lemma \ref{L:crossblock} that $c(\tau)=c(\sigma)$ and $c^*(\tau')=c^*(\sigma')$. Writing $\tau=\tau_n\sqcup\cdots\sqcup\tau_1$, we therefore have
		\[ \eta(\tau)=c(\sigma)+\sum_{j\ge 1}\eta(\tau_j)\quad \mbox{ and }\quad\eta^*(\tau')=c^*(\sigma')+\sum_{j\ge 1}\eta(\tau_j),\]
		where the second equality uses $\eta^*(\tau_j')=\eta(\tau_j)$ for each rectangular tableau. Hence the assumed identities for rectangles give
		\[
		\begin{aligned}
			\sum_{\tau\in\CMcal{G}(\sigma)}q^{\eta(\tau)}
			=q^{c(\sigma)}
			\prod_{j\ge 1}
			\left(
			\sum_{\tau_j\in\CMcal{G}(\sigma_j)}
			q^{\eta(\tau_j)}
			\right)
			=
			q^{c(\sigma)+\sum_{j\ge 1}\eta(\sigma_j)}
			\prod_{j\ge 1}d_\varepsilon(\sigma_j)
			=
			q^{\eta(\sigma)}d_\varepsilon(\sigma).
		\end{aligned}
		\]
		Using the decomposition of $\eta^*(\tau')$ in the same way yields
		\[
		\sum_{\tau\in\CMcal{G}(\sigma)}q^{\eta^*(\tau')}
		=q^{\eta^*(\sigma')}d_\varepsilon(\sigma).
		\]
		Multiplying these identities by $t^{\maj(\sigma)}$ and
		$t^{\maj(\sigma')}$, respectively, proves the result.
	\end{proof}
			
	\begin{lemma}\label{L:dec1}
		For every filling $\tau$, there is a unique canonical filling
		$\sigma$ such that $\tau\in\CMcal{G}(\sigma)$.
		Likewise, there is a unique dual canonical filling $\bar\sigma$
		such that $\tau\in\CMcal{G}(\bar\sigma)$.
	\end{lemma}
	\begin{proof}
		
		We prove only the case where $\sigma$ is a canonical rectangular tableau, since the other dual case is analogous.
		We start from the top row of $\tau$ and sort the entries of the top row into weakly decreasing order. Suppose that the top $n-r$ rows of $\tau$ form a canonical filling; we construct a unique canonical tableau $\sigma$ of the top $n-r+1$ rows as follows.
		
		For every neutral block $P$ in row $r$ of $\Delta_{\pi^{\bullet}}(\tau\vert_1^{r+1})$, let $b$ be its common upper entry.  
		Then we rearrange the elements of $P$ into a new sequence $w$ with respect to $b$ as follows: move all entries that are less than $b$ to the front in weakly decreasing order, followed by entries that are larger than or equal to $b$ in the left-to-right order of $P$. 
		
		This sorting process for neutral blocks is implemented as follows: Take $\tilde{u}$ to be the shortest permutation of $[k]$ that sorts all neutral blocks $P$ as above.
		Since the relative order of entries in the non-descent columns formed by each neutral block and its upper neighbor is preserved by $\tilde{u}$, $\tilde{u}$ is also the shortest permutation that transforms $\alpha(P)$ to $\alpha(w)$. After this step, the descent entries in each neutral block occur first in weakly decreasing order, while the relative order of the non-descent entries is preserved. Let $\hat{\tau}$ be the filling obtained after this process on $\tau$. 
		
       Let $p=(p_1,\ldots,p_k)=\hat{\tau}\vert_r^{r}$ with upper neighbor $a=(a_1,\ldots,a_k)$. The sequences $a$ and $p$ determine $\pi^{\diamond}$.
		For the unique non-descent block $Q$ of the $r$th row of $\Delta_{\pi^{\diamond}}(\hat{\tau}\vert_1^{r+1})$, we sort the entries of $Q$ into a weakly decreasing sequence. Let $\tilde{v}$ be the shortest permutation of $[k]$ to achieve this. The block $Q$ after the rearrangement by $\tilde{v}$ is still a non-descent block because the neighboring entries $c,d$ of $Q$ are swapped by $\tilde{v}$ precisely when $c<d$, that is, $d>c\ge e\ge f$ for the quadruple $(e,f,c,d)$ of two neighboring non-descent columns. Here $e\ge f$, since the upper neighbor of $Q$ weakly decreases.
		$$\begin{ytableau}
			e  & f   \\
			c  & d 
		\end{ytableau}
		\, \rightarrow 
		\begin{ytableau}
			e  & f   \\
			d  & c 
		\end{ytableau}
		\, .$$
		Consequently Lemma \ref{L:equ2Def} assures the non-descent block in row $r$ of $\delta_{\tilde{v}}^r(\hat{\tau})$ fulfills condition (II) of canonical tableaux in Definition \ref{Def:3}. 
		
		Denote by $\sigma$ the canonical filling obtained after the entire sorting process on $\tau$. Reversing the process shows that $\tau\in \CMcal{G}(\sigma)$ by Definition \ref{D:family}. For a filling $\tau$ of general shape, applying the procedure to each rectangular component $\tau_i$ independently. Note that the second step ensures both (I) and (II). This normalization uniquely reverses the two rearrangements at each row, recovering the canonical representative and all intermediate choices.  Therefore, we conclude (\ref{E:g11}).

	\end{proof}

	\begin{example}
		Let $\lambda=(5,5,5)$ and consider the filling $\tau$ below.
		Here only the non-descent blocks need to be sorted.
		The top row is sorted by $\delta_3\,\delta_2$,
		and the non-descent block $(4,3,7)$ in the second row
		are then sorted by $\delta_3^2\,\delta_4^2$.
		For the bottom row, $\pi^\diamond=s_3s_2s_1$ brings columns
		$1$ and $5$ together; $\delta_{s_4}^{1\,\diamond}$ changes
		their lower entries from $(3,8)$ to $(8,3)$ and restores
		the second row. Thus
		\begin{gather*}
			\tau=
			\begin{ytableau}
					*(green)5&	*(green)1&	*(green)3&	*(green)2&	*(green)1\\
				3&3&2&4&7\\
				3&8&1&2&2
			\end{ytableau}
			\xrightarrow{\delta_3\delta_2}
			\begin{ytableau}
				5&3&2&1&1\\
					*(green)3&	*(green)2&	*(green)4&	*(green)3&	*(green)7\\
				3&1&2&8&2
			\end{ytableau}
			\\[4pt]
			\qquad\xrightarrow{\delta_3^2\delta_4^2}
			\begin{ytableau}
				5&3&2&1&1\\
				3&2&7&4&3\\
					*(green)3&	*(green)1&	*(green)2&	*(green)2&	*(green)8
			\end{ytableau}
			\xrightarrow{\delta_{s_4}^{1\,\diamond}}
			\begin{ytableau}
				5&3&2&1&1\\
				3&2&7&4&3\\
				8&1&2&2&3
			\end{ytableau}
			=\sigma.
		\end{gather*}
		The filling $\sigma$ satisfies conditions~(I) and~(II)
		and is the unique canonical filling with
		$\tau\in\CMcal{G}(\sigma)$.
	\end{example}
	
	We now define the canonical coefficient. Let $N$ be the largest entry of a filling $\sigma$. For each rectangular component $\sigma_j$ in the decomposition $\sigma=\sigma_n\sqcup \cdots \sqcup \sigma_1$, let $\nu_{i,j}^0=0$ and 
	\[
	\nu_{i,j}^k=\#\{\text{entries in row }i\text{ of }\sigma_j\text{ that are at least }N+1-k\}.
	\]
	For two adjacent rows $i+1$ and $i$ of $\sigma_j$, consider the vertical pairs $(a,b)$ and set
	\[
	s_h^k=\#\{(a,b)\text{ in these rows}:a=N+1-k,\ b\ge N+1-h\}
	\]
	for $1\le k\le h\le N$. Therefore, $s_k^k$ counts the non-descent columns with upper entry $N+1-k$, whereas $s_h^k-s_{h-1}^k$ counts the descent columns with the same upper entry $N+1-k$ and lower entry $N+1-h$ for $h>k$. In particular,
	$s_N^k=\nu_{i+1,j}^k-\nu_{i+1,j}^{k-1}$ counts the occurrences of $N+1-k$ in row $i+1$ of $\sigma_j$, so \eqref{E:defs}, \eqref{eq7}, and \eqref{eq8} are satisfied. 
	Define
	\begin{align}
		d_\diamond(\sigma)&=\prod_{j=1}^nd_\diamond(\sigma_j),\label{E:coeffd}\\
		d_\diamond(\sigma_j)&=\prod_{i=1}^j\prod_{k=1}^{N-1}
		{\nu_{i,j}^{k+1}-s_{k+1}^{1,k}\brack\nu_{i,j}^{k}-s_k^{1,k}}_q
		\prod_{h=1}^{k}{s_{k+1}^h\brack s_k^h}_q\label{E:coeffd2}
	\end{align}
	where $d_{\diamond}(\sigma_i)=1$ if $i$ is not a part of $\lambda'$.
	
	We now introduce some total orderings related to inversions and prove a sequence of lemmas towards the proof of (\ref{E:count11}).
	\subsection{Local $q$-weights and Gaussian coefficients}
	Following \cite{LN12}, we use inversions with respect to arbitrary total orders. Their multiset generating functions are Gaussian multinomials.
	
	\begin{definition}[$\prec$-inversions \cite{LN12}]\label{Def:generlization_inv}
		For a total order $\prec$ on the set $[N]$, let $\inv(w,\prec)$ count pairs $i<j$ with $w_i\succ w_j$. We use the cyclic order
		\[
		u+1\overset{u}{<}\cdots\overset{u}{<}N\overset{u}{<}1\overset{u}{<}\cdots\overset{u}{<}u
		\]
		for all $u\in [N]$, where $u$ is the largest entry by $\overset{u}{<}$. In particular, the usual total ordering $<$ is the special case $u=N$. Define the opposite order $\overset{u}{\lessdot}$ by $x\overset{u}{\lessdot}y$ if and only if $y\overset{u}{<}x$. We set $\coinv(w,\overset{u}{<})=\inv(w,\overset{u}{\lessdot})$.
	\end{definition}
	
	\begin{lemma}[\cite{FH,LN12}]\label{L:length_inv}
		Let $p$ be a word whose distinct entries have multiplicities $m_i$ for $1\le i\le r$, with sum $n$. For any total order $\prec$,
		\[
		\sum_{w\in\mathrm{Sym}(p)}q^{\inv(w,\prec)}
		={n\brack m_1,\ldots,m_r}_q.
		\]
		If $p$ is weakly increasing with respect to $\prec$, then 
		\[
		\sum_{\tilde w\in\widetilde{\mathrm{Sym}}(p)}q^{\ell(\tilde w)}
		={n\brack m_1,\ldots,m_r}_q.
		\]
	\end{lemma}
	
	\begin{lemma}
		Let $1\le a\le N+1$. Suppose $p$ lists its entries less than $a$ in weakly decreasing order, followed by its entries at least $a$ in weakly decreasing order. Then, for every $w\in\mathrm{Sym}(p)$,
		\begin{equation}\label{E:lp}
			L_p(w)=\inv(w,\overset{a-1}{\lessdot}).
		\end{equation}
	\end{lemma}
	\begin{proof}
		Write $\prec$ for $\overset{a-1}{\lessdot}$ and let
		$I(w)=\inv(w,\prec)$. Since $p$ is weakly increasing with respect to $\prec$,
		we have $I(p)=0$. An adjacent transposition changes the
		inversion number by at most one, thus $L_p(w)\ge I(w)$.
		
		Conversely, repeatedly exchanging adjacent inverted pairs sorts $w$ into $p$. Each exchange decreases the inversion number by exactly one, so this requires exactly $I(w)$ steps. Reversing these exchanges gives a sequence of $I(w)$ adjacent transpositions taking $p$ to $w$, hence $L_p(w)\le I(w)$.
		The two inequalities prove the claim.
		
	\end{proof}
	

	\begin{lemma}\label{L:two_row}
		Let a neutral block $p$ have upper entry $a\ge1$, with its descent entries first in weakly decreasing order. Let $w$ rearrange $p$ while preserving the relative order of its non-descent entries. For the two-row fillings $\sigma$ and $\tau$ with lower words $p$ and $w$,
		\begin{equation}\label{E:symquinv}
			L_p(w)=\quinv(\tau)-\quinv(\sigma)=\eta(\tau)-\eta(\sigma).
		\end{equation}
	\end{lemma}
	\begin{proof}
		Equal upper entries make the local contribution to $\eta$ equal to the number of queue inversion triples. Pairs of non-descent entries retain their relative order, so their contributions to the difference $\quinv(\tau)-\quinv(\sigma)$ cancel. The remaining difference is
		\[
		\quinv(\tau)-\quinv(\sigma)=\sum_{i<j}\chi(w_j<a\le w_i)+\sum_{i<j}\chi(w_i<w_j<a)
		=\inv(\alpha(w),\overset{a-1}{\lessdot}).
		\]
		Let $\tilde w$ be the shortest permutation transforming $\alpha(p)$ to $\alpha(w)$. Then
		$\tilde w$ is also the shortest permutation sending $p$ to $w$, thus obtaining
		$L_p(w)=L_{\alpha(p)}(\alpha(w))=\ell(\tilde w)$, and \eqref{E:lp} applies to $\alpha(p)$.
	\end{proof}
		
	\begin{lemma}\label{L:symQ1}
		Let $\sigma$ and $\tau$ be two-row fillings with common upper
		row $a=(a_1,\ldots,a_k)$ and lower rows $p$ and $w$,
		respectively, where $w\in\mathrm{Sym}(p)$,
		$a_i\le p_i$ and $a_i\le w_i$ for every $i$.
		If the columns $(a_i,p_i)$ satisfy condition~(II)
		of Definition~\ref{Def:3}. 
		Then for $\eta\in\A^+$ with $(v>u\ge z>w)\in\S\cap a_2$, we have
		\begin{equation}\label{E:symquinv2}
			\eta(\tau)-\eta(\sigma)
			= L_{\pi^\diamond(p)}\bigl(\pi^\diamond(w)\bigr) 
			=\coinv\bigl(\pi^\diamond(w)\bigr).
		\end{equation}
	\end{lemma}
	\begin{proof}
		
       We note that $ L_{\pi^\diamond(p)}\bigl(\pi^\diamond(w)\bigr)= \coinv\bigl(\pi^\diamond(w)\bigr)$. Indeed,
       this equality follows from Remark~\ref{Rem:5} and Lemma \ref{L:equ2Def}. 
       
       Since $\sigma$ and $\tau$ have the same upper row, the
       $\S$-quadruples involving the adjoined zero make no contribution to
       $\eta(\tau)-\eta(\sigma)$. We may therefore sort the columns of both fillings by
       $\pi^\diamond$: columns with equal upper entries retain
       their relative order, while exchanging the order of two
       columns with distinct upper entries preserves their
       membership in $\S$ by Definition~\ref{Def:typeqinv} (1). Write
       \[
       \mathcal{A}=\pi^\diamond(a),\qquad
       \CMcal{P}=\pi^\diamond(p),\qquad
       \CMcal{W}=\pi^\diamond(w).
       \]
       Both $\mathcal{A}$ and $\CMcal{P}$ are weakly decreasing by Lemma \ref{L:equ2Def}. For $\CMcal{V}=\CMcal{P}$ or $\CMcal{W}$, we have $\CMcal{V}_i\ge \mathcal{A}_i$ for every $i$.
       For $i<j$, these inequalities and the assumption
       $(v>u\ge z>w)\in\S\cap a_2$ give, by
       Definition~\ref{Def:typeqinv},
       \begin{align}\label{E:Q1}
       \Q_{\S}(\mathcal{A}_i,\mathcal{A}_j,\CMcal{V}_i,\CMcal{V}_j)
       =\chi(\CMcal{V}_i<\CMcal{V}_j)+\chi(\mathcal{A}_i>\CMcal{V}_j).
       \end{align}
       For fixed $i$, if $j\le i$, then $\CMcal{V}_j\ge \mathcal{A}_j\ge \mathcal{A}_i$. Hence $\chi(\mathcal{A}_i>\CMcal{V}_j)=1$ may hold only if $\CMcal{V}_j$ is located to the right of position $i$. In other words,
       \[
       \sum_{j>i}\chi(\mathcal{A}_i>\CMcal{V}_j)=\#\{j\in[k]:\CMcal{V}_j<\mathcal{A}_i\}.
       \]
       The right-hand side is the same for $\CMcal{P}$ and $\CMcal{W}$, since they have the same multiset of entries. Thus the second terms on the right-hand side of (\ref{E:Q1}) cancel when we subtract
       the two quadruple counts. Since $P$ is weakly decreasing, i.e., $\coinv(P)=0$, 
       the remaining terms yield
       \begin{equation}\label{E:coinv1}
       	\eta(\tau)-\eta(\sigma)
       	=\coinv(\CMcal{W})
       	=\coinv\bigl(\pi^\diamond(w)\bigr).
       \end{equation}
      This completes the proof.
   \end{proof}
	

	
	
	\subsection{Proof of Theorem \ref{P:can11}}
	To complete the proof of Theorem \ref{P:can11}, it suffices, by Lemmas \ref{L:rec} and \ref{L:dec1}, to prove (\ref{E:count11}) for rectangular fillings. We begin with Propositions \ref{prop:q1} and \ref{prop:q2} for canonical tableaux.

	\begin{proposition}\label{prop:q1}
		Let $P$ be a neutral block of a canonical filling,
		with common upper entry $N+1-h$. In the notation of
		\eqref{E:coeffd2}, we have
		\begin{equation}\label{E:alphap1}
			\sum_{\tilde w\in\widetilde{\mathrm{Sym}}(\alpha(P))}q^{\ell(\tilde w)}
			={s_N^h\brack s_{h+1}^h-s_h^h,\ldots,s_N^h-s_{N-1}^h,s_h^h}_q.
		\end{equation}
	\end{proposition}
	\begin{proof}
		The block $P$ has $s_N^h$ entries. In $\alpha(P)$, all entries
		of $P$ at least $N+1-h$ are replaced by $N+1-h$, which
		therefore occurs $s_h^h$ times. Each smaller value $N+1-m$,
		where $h<m\le N$, retains its multiplicity
		$s_m^h-s_{m-1}^h$. With these multiplicities, Lemma~\ref{L:length_inv} 
		gives the stated formula.		
	\end{proof}
	
	\begin{proposition}\label{prop:q2}
		Let $P$ be the non-descent block in row $i$ of a canonical rectangular filling $\sigma_j$, with upper word $a$. Then
		\begin{equation}\label{E:beta1}
			\sum_{\substack{w\in\mathrm{Sym}(P)\\ w\ge a}}q^{L_{\pi^{\diamond}(P)}(\pi^{\diamond}(w))}
			=\prod_{k=1}^{N-1}{\nu_{i,j}^{k+1}-s_{k+1}^{1,k}\brack\nu_{i,j}^{k}-s_k^{1,k}}_q,
		\end{equation}
	    where $w\ge a$ means that $w_r\ge a_r$ for all $r\ge 1$.
	\end{proposition}
	\begin{proof}
		Sort the columns and write $\mathcal{A}=\pi^\diamond(a)$ and $\CMcal{W}=\pi^\diamond(w)$. Then by \eqref{E:symquinv2},
		\begin{equation}\label{E:beta2}
			\sum_{\substack{w\in\mathrm{Sym}(P)\\w\ge a}}q^{L_{\pi^{\diamond}(P)}(\pi^{\diamond}(w))}
			=\sum_{\substack{w\in\mathrm{Sym}(P)\\w\ge a}}q^{\coinv(\CMcal{W})}.
		\end{equation}
	    We next represent $\CMcal{W}$ by a collection of binary words whose queue-inversion numbers sum to $\coinv(\CMcal{W})$. 
	    This will express the generating function as a product of Gaussian binomial coefficients. 
	    
	    For each distinct value $x$ occurring in $\CMcal{W}$, take the positions $i$ such that $\mathcal{A}_i\le x\le \CMcal{W}_i$, retain entries $\CMcal{W}_i$ equal to $x$, and replace larger entries by $x+1$. Call the resulting binary word $b(x)$. It has length 
	    \[
	    H_x=\#\{i:A_i\le x\}-\#\{\text{entries of }P\text{ less than }x\}.
	    \]
	    If $m_x$ is the multiplicity of $x$ in $P$, then $b(x)$ has $m_x$ copies of $x$ and $H_x-m_x$ copies of $x+1$, independently of $\CMcal{W}$.
	    
	    Conversely, the words $b(x)$ uniquely determine $\CMcal{W}$. Suppose that all entries smaller than $x$ have been placed.
	    There are then $H_x$ unfilled positions whose upper entries are at most $x$. Read these positions from left to right and match them with the letters of $b(x)$: a letter $x$
	    specifies a position to be filled with $x$, while a letter $x+1$ leaves that position for a larger entry.
	    Thus every choice of the words $b(x)$ with the prescribed multiplicities gives a unique word $\CMcal{W}$ satisfying $\CMcal{W}_i\ge \mathcal{A}_i$.
	    
	    Now consider a pair $r<s$ with $\CMcal{W}_r=x<\CMcal{W}_s$. Since the upper row is weakly decreasing, we have $\mathcal{A}_s\le \mathcal{A}_r\le x$.
	    Both positions therefore occur in $b(x)$, where they form a pair $(x,x+1)$. This pair is counted only in $b(x)$, and every pair counted by $\quinv(b(x))$ arises in this way.
	    Hence
	    \[
	    \coinv(\CMcal{W})=\sum_{x\in \CMcal{W}}\coinv(b(x)).
	    \]
	   For $x=N-k$, the number $H_x$ counts columns $i$ in the
	   two-row filling $\sigma_j\vert_{i}^{i+1}$ such that $a_i\le x\le P_i$. The number $H_x-m_x$ counts
	   those whose lower entry $P_i$ with $P_i>x$.
	   Now applying Lemma~\ref{L:length_inv} to each binary word $b(x)$ and using the definitions of $\nu$ and $s$, we obtain
	   \begin{equation}\label{E:beta4}
	   	\sum_{\substack{w\in\mathrm{Sym}(P)\\w\ge a}} q^{\coinv(\CMcal{W})}=\prod_{x\ge 1} {H_x\brack H_x-m_x}_q
	   	=\prod_{k=1}^{N-1}{\nu_{i,j}^{k+1}-s_{k+1}^{1,k} \brack \nu_{i,j}^{k}-s_k^{1,k}}_q.
	   \end{equation}
	\end{proof}
	
	We can now complete the proof of Theorem~\ref{P:can11}.
	
	\begin{proof}[Proof of Theorem \ref{P:can11}]
		Equation \eqref{E:g11} is proved in Lemma~\ref{L:dec1}. By Lemma~\ref{L:rec}, it remains to establish \eqref{E:count11} and \eqref{E:count13} for rectangles. We first prove \eqref{E:count11}. Let $\sigma$ be a canonical filling of shape $(k^n)$.
		Before row $r$ is rearranged, rows $r$ and $r+1$ agree with those of $\sigma$, since the preceding operations affect only rows below $r$.
		
		After applying $\pi^\diamond$, we keep the upper row fixed and rearrange the non-descent entries, with each column remaining a non-descent.
		The mixed-count invariance established in the proof of Proposition~\ref{prop:thm9rec} (see~\eqref{eqpro1A}) shows that
		$\pi^\diamond$ preserves the number of $\S$-quadruples formed by one descent column and one non-descent column.
		
		For a neutral block, applying $\pi^\bullet$ places its columns consecutively. Every other column lies to one side of the block, so the number of $\S$-quadruples formed with that column depends only on the block's common upper entry and the multiset of its lower entries. This number is unchanged by a rearrangement within the block.
		
		Now let $\tau\in\CMcal{G}(\sigma;r-1)$, and let $\tau'$ be obtained from $\tau$ by the two rearrangements in row $r$.
		Write $Q$ for the non-descent block, $a$ for its upper word, and $w$ for its chosen rearrangement. After this operation,
		let $P_1,\ldots,P_m$ be the neutral blocks and $w_1,\ldots,w_m$ their chosen rearrangements.
		Lemmas~\ref{L:symQ1} and~\ref{L:two_row} give
		\[
		\eta(\tau')-\eta(\tau)
		=
		L_{\pi^{\diamond}(Q)}(\pi^{\diamond}(w))
		+\sum_{h=1}^{m}L_{P_h}(w_h).
		\]
		Indeed, Lemma~\ref{L:delta} preserves the number of $\S$-quadruples between rows $s$ and $s+1$ for $1\le s<r$. Even though row $r+1$ is temporarily permuted, its original order is restored at the end. The number of $\S$-quadruples  between rows $r+1$ and $r+2$ is also unchanged.
		
		The rearrangement of $Q$ changes neither the descent entries nor the positions of the non-descent entries in row $r$.
		Since $\alpha$ replaces each non-descent entry by its upper neighbor, the words $\alpha(P_h)$ are independent of the choice of $w$. Propositions~\ref{prop:q1} and~\ref{prop:q2} therefore give the same counting polynomial $B_r(q)$ for every $\tau\in\CMcal{G}(\sigma;r-1)$, where $B_r(q)$ is the product of the right-hand side of \eqref{E:beta1} for $Q$ and the right-hand sides of \eqref{E:alphap1} for all neutral blocks $P_h$ in row $r$, namely, 
		\begin{align*}
			B_r(q)&=
			\prod_{k=1}^{N-1}  \begin{bmatrix} \nu_{r,n}^{k+1}-s_{k+1}^{1,k} \\ \nu_{r,n}^{k}-s_{k}^{1,k} \end{bmatrix}_{q}\prod_{h=1}^{N-1}
			\begin{bmatrix} s_{N}^{h} \\ s_{h+1}^{h}-s_{h}^{h},\ldots,s_{N}^{h}-s_{N-1}^{h},s_{h}^{h}
			\end{bmatrix}_{q}\\
		&=\prod_{k=1}^{N-1}  \begin{bmatrix} \nu_{r,n}^{k+1}-s_{k+1}^{1,k} \\ \nu_{r,n}^{k}-s_{k}^{1,k} \end{bmatrix}_{q}\prod_{h\le k}{s_{k+1}^h \brack s_k^h}_q\,.
		\end{align*}
		By Lemma~\ref{L:dec1}, each filling in $\CMcal{G}(\sigma;r)$ arises from a unique preceding filling in $\CMcal{G}(\sigma;r-1)$ and a unique choice of rearrangements. Summing over these choices gives
		\begin{equation}\label{E:rec2}
			\sum_{\tau'\in\CMcal{G}(\sigma;r)}q^{\eta(\tau')}
			=
			B_r(q)
			\sum_{\tau\in\CMcal{G}(\sigma;r-1)}q^{\eta(\tau)}.
		\end{equation}
		Multiplying \eqref{E:rec2} for $1\le r\le n$ yields $\sum_{\pi\in\CMcal{G}(\sigma)}q^{\eta(\pi)}=q^{\eta(\sigma)}d_\diamond(\sigma)$.
		The top-row factor is the usual multiset coefficient, since its upper entries are all zero and all $s$-sequences vanish. The major index is constant by Lemma~\ref{L:rec}. This proves \eqref{E:count11}.
		
		For dual canonical fillings, the proof is analogous, with the descent block replacing the non-descent block and $\pi^\dagger$ replacing $\pi^\diamond$. Within each neutral block, we use $\bar\alpha$ instead of $\alpha$ and preserve the relative order of the descent entries.
		For each rectangular component $\sigma_j$, put
		\[
		\nu_{i,j}^k
		=\#\{\text{entries in row }i\text{ of }\sigma_j
		\text{ at most }k\},
		\qquad \nu_{i,j}^0=0.
		\]
		For each pair of adjacent rows, let
		\[
		s_h^k
		=\#\{\text{columns }(a,b):a=k+1,\ b\le h\},
		\qquad 0\le k<N,\quad k\le h\le N.
		\]
		These are the arrays in \eqref{E:defs2}. Multiplying the factors for the descent and neutral blocks gives
		\begin{align}
			d_\dagger(\sigma)
			&=\prod_{j=1}^{n}d_\dagger(\sigma_j),
			\label{E:coeffdag}\\
			d_\dagger(\sigma_j)
			&=\prod_{i=1}^{j}\prod_{k=0}^{N-1}
			\begin{bmatrix}
				\nu_{i,j}^{k+1}-s_{k+1}^{0,k}\\
				\nu_{i,j}^{k}-s_k^{0,k}
			\end{bmatrix}_q
			\prod_{h=0}^{k}
			\begin{bmatrix}
				s_{k+1}^{h}\\
				s_k^{h}
			\end{bmatrix}_q.
			\label{E:coeffdag2}
		\end{align}
		An empty component has coefficient $1$. The recurrence \eqref{E:rec2} now proves \eqref{E:count13},
		and Lemma~\ref{L:rec} yields the remaining identities.
	\end{proof}
	
	\section{Proof of Theorem \ref{T:mono1}}\label{S:thm41}
	We enumerate canonical and dual canonical tableaux to prove \eqref{eq6}--\eqref{eq64} of Theorem~\ref{T:mono1}. 
	Throughout this section, $1\le i\le j\le n$, $m_j=\lambda_j-\lambda_{j+1}$, and $s=s(i,j)$ denotes the array associated with rows $i+1$ and $i$ of $\sigma_j$. In sums over $(i,j)$, $s$ denotes the corresponding array. 
	
	\begin{lemma}\label{L:sigmaj}
		Fix a positive integer $j$. The sequences $\{\nu_{i,j}\}_{i\le j}$ satisfying \eqref{eq7}--\eqref{eq8}, together with an array $s(i,j)$ satisfying \eqref{E:defs} for $\nu=\nu_{i+1,j}$ at each level, determine a unique canonical tableau $\sigma_j$, provided that the product in \eqref{E:coeffd2} is nonzero.
	\end{lemma}
	\begin{proof}
		
		We prove the statement by constructing such a canonical tableau. Fill the top row in weakly decreasing order, with $\nu_{j,j}^k-\nu_{j,j}^{k-1}$ copies of $N+1-k$. Suppose that rows $i+1,\ldots,j$ have been filled. Below entries $N+1-h$, first place $s_{k+1}^h-s_k^h$ copies of $N-k$, for $h\le k<N$, in weakly decreasing order in the leftmost available positions. These are precisely the descent columns. The remaining $s_h^h$ positions below $N+1-h$ must be non-descents.
		
		The unused entries consist of $\nu_{i,j}^1$ copies of $N$ and
		\begin{equation}\label{E:v1}
			\nu_{i,j}^{k+1}-\nu_{i,j}^{k}-\sum_{h=1}^k(s_{k+1}^h-s_k^h)
		\end{equation}
		copies of $N-k$, for $1\le k<N$. Arrange these elements in weakly decreasing order and insert the $m$th
		largest element into the empty box whose upper neighbor is the $m$th largest. This matching produces non-descents because, for $0\le k<N$, the number of unused entries at least $N-k$ is 
		\begin{equation}\label{E:v2}
			\nu_{i,j}^1+\sum_{h=1}^{k}( \nu_{i,j}^{h+1}-\nu_{i,j}^{h}-\sum_{m=1}^{h} (s_{h+1}^m-s_h^m) )
			=\nu_{i,j}^{k+1}-s_{k+1}^{1,k}+\sum_{h=1}^k s_h^h
			\ge\sum_{h=1}^{k+1}s_h^h.
		\end{equation}
		For $1\le k<N-1$, the last inequality follows from $\nu_{i,j}^{k+1}\ge s_{k+1}^{1,k+1}$, which is guaranteed by the nonvanishing of the Gaussian $q$-binomials in \eqref{E:coeffd2}. For $k=N-1$, it follows from the equality of the row widths. This yields a unique canonical tableau, constructed row by row from top to bottom, as desired.
		
	\end{proof}
    Let $\sigma(\nu,s)$ denote the unique tableau obtained from the sequences $\nu$ and $s$ in Lemma \ref{L:sigmaj}. For any filling $\sigma$, let $\bar\eta(\sigma)$ count the $\S$-quadruples and triples not counted by $\eta$. Each pair of boxes in one row determines exactly one such quadruple or triple, including those with a zero upper entry. Hence
	\begin{equation}\label{E:nlambda}
		\bar\eta(\sigma)+\eta(\sigma)=n(\lambda')=\sum_i\binom{\lambda_i}{2}.
	\end{equation}
    By symmetry, the coefficient of $m_\mu$ can be computed from the single monomial $x_N^{\mu_1}\cdots x_1^{\mu_N}$. Equations \eqref{E:g11}, \eqref{eq61}, and Lemma \ref{L:sigmaj} therefore give
    \begin{equation}\label{E:cri1}
    	\CMcal{P}_{\lambda\mu}(q,t)
    	=\sum_{\nu,s}\sum_{\tau\in\CMcal{G}(\sigma(\nu,s))}
    	t^{\maj(\tau)}q^{\bar\eta(\tau)},
    \end{equation}
   where \eqref{eq7}--\eqref{eq9} and \eqref{E:defs} are satisfied. For fixed $\nu$, the class of fillings row-equivalent to $\sigma$ in each rectangle is the disjoint union of the corresponding families:
\begin{equation}\label{E:cri2}
	\sum_{\tau\in\langle\sigma\rangle}t^{\maj(\tau)}q^{\bar\eta(\tau)}
	=\sum_s\sum_{\tau\in\CMcal{G}(\sigma(\nu,s))}t^{\maj(\tau)}q^{\bar\eta(\tau)}.
\end{equation}
Therefore \eqref{E:eq162} makes this sum independent of the choice of $\eta\in\A^+$. We use $\S=\S_2$, for which \eqref{E:count11} applies.


Let $\overline{\quinv}(\sigma)$ denote the number of non-queue-inversion triples between columns of unequal heights. Then
\begin{equation}\label{E:quinvc}
	\bar\eta(\sigma)=\sum_{i\le j}\bar\eta(\sigma_j\vert_i^{i+1})+\overline{\quinv}(\sigma).
\end{equation}
		
	
	\begin{theorem}\label{T:count2}
		For the canonical tableau $\sigma=\sigma(\nu,s)=\sigma_n\sqcup \cdots \sqcup \sigma_1$ determined by Lemma \ref{L:sigmaj}, $\S=\S_2$ and $1\le i\le j\le n$,  we have
		\begin{align}
			\maj(\sigma_j\vert_{i}^{i+1})&=(j-i)\sum_{k=1}^{N-1}(s_N^k-s_k^k),\label{E:majsigma1}\\
			\bar\eta(\sigma_j\vert_{i}^{i+1})&=\sum_{k=1}^N\binom{\nu_{i,j}^k-\nu_{i,j}^{k-1}}2+\xi_1(s,\nu_{i,j})+\sum_{k=1}^{N-1}\sum_{h=1}^k(s_{k+1}^h-s_k^h)s_k^h,\label{E:gwsigma1}\\
			&\quad+\sum_{k=1}^{N-1}(\nu_{i,j}^{k+1}-s_{k+1}^{1,k}-\nu_{i,j}^k+s_k^{1,k})(\nu_{i,j}^k-s_k^{1,k})\notag\\
			\overline{\quinv}(\sigma)&=\sum_{1\le i\le j\le n}(\lambda_i-\lambda_j)\sum_{k=1}^{N-1}(s_N^k-s_k^k)\notag\\
			&\quad+\sum_{1\le i\le j<\ell\le n}\sum_{k=1}^N
			(\nu_{i,j}^k-\nu_{i,j}^{k-1})(\nu_{i,\ell}^k-\nu_{i+1,\ell}^{k-1}).\label{E:gwsigma2}
		\end{align}
	\end{theorem}
	\begin{proof}
		For a vertical pair $(z,u)$ of $\sigma_j\vert_i^{i+1}$, $s_k^h$ counts the occurrences of $z=N+1-h$ with $u\ge N+1-k$. Consequently,
		\begin{equation}\label{E:descol1}
			\#\{(z,u):z>u\}=\sum_{k=1}^{N-1}(s_N^k-s_k^k).
		\end{equation}
		Each such descent has leg-length $j-i-1$, thus proving \eqref{E:majsigma1}. Moreover, 
		\begin{itemize}
			\item $s_{k+1}^h-s_k^h$ counts the columns with $z=N+1-h$ and $u=N-k$; 
			\item $\nu_{i,j}^k-s_k^{1,k}$ counts the columns with $z\le N-k<u$;
			\item $\nu_{i,j}^{k+1}-s_{k+1}^{1,k}-\nu_{i,j}^k+s_k^{1,k}$ counts the columns with $z\le u=N-k$.
		\end{itemize}
		These partition the non-$\S_2$ quadruples into four disjoint classes. By definition,
		\begin{equation}\label{E:chi1}
			\#\{(z,w,u,v):u=v\}=\sum_{k=1}^N\binom{\nu_{i,j}^k-\nu_{i,j}^{k-1}}2.
		\end{equation}
	    	Furthermore, $s_{k}^{h}(s_{k+1}^{h}-s_{k}^{h})$ counts the quadruples $(z,w,u,v)$ such that $z=w=N+1-h>u>v=N-k$ or $v\ge z=w=N+1-h>u=N-k$. Therefore, 
	    \begin{align*}
	    	\#\{(z,w,u,v):z=w>u>v\,\textrm{ or }\,v\ge z=w>u\}=\sum_{k=1}^{N-1}\sum_{h\le k}(s_{k+1}^{h}-s_{k}^{h})s_{k}^{h}
	    \end{align*}
	    and all these quadruples are not $\S_2$-quadruples. Since 
	    \begin{align*}
	    	(\nu_{i,j}^{k}-s_{k}^{1,k}) (\nu_{i,j}^{k+1}-s_{k+1}^{1,k}-\nu_{i,j}^{k}+s_{k}^{1,k})
	    \end{align*}
	    counts the quadruples $(z,w,u,v)$ such that $u>v=N-k\ge z\ge w$ or $v>u=N-k\ge w>z$, we obtain
	    \begin{align*}
	    	&\quad \#\{(z,w,u,v):u>v\ge z\ge w\,\textrm{ or }\,v>u\ge w>z\}\\
	    	&=\sum_{k=1}^{N-1}(\nu_{i,j}^{k}-s_{k}^{1,k}) (\nu_{i,j}^{k+1}-s_{k+1}^{1,k}-\nu_{i,j}^{k}+s_{k}^{1,k}).
	    \end{align*}
        Finally, by definition, 
        \begin{align*}
        	\xi_1(s,\nu_{i,j})
        	=\sum_{1\le h\le k<N}(s_{k+1}^h-s_k^h)(\nu_{i,j}^k-s_k^{h,k})
        \end{align*}
        counts the classes
        \[
        \begin{gathered}
        	z>u\ge w,\ v>u;\qquad w>v\ge z,\ u>v;\\
        	w>z>u,\ v>u;\qquad z>w>v,\ u>v.
        \end{gathered}
        \]
        These are precisely the remaining non-$\S_2$ orderings compatible with Definition \ref{Def:3}. This proves \eqref{E:gwsigma1}, including the top row, where all $s_k^h$ vanish. We next count the non-$\quinv$ triples between $\sigma_j$ and a different $\sigma_{\ell}$ to prove \eqref{E:gwsigma2}.
        
        Considering all possible relative orders of $z$, $u$, and $v$, (\ref{E:Qindicator}) gives
        \begin{equation}\label{eqthm2}
        	1-\Q(z,u,v)=\chi(z>u)+\chi(u\ge v)-\chi(z>v)
        \end{equation}
        For a triple, a descent in $\sigma_j\vert_i^{i+1}$ can be adjoined to any of $\lambda_i-\lambda_j$ possible right-hand entries in row $i$ of some $\sigma_{\ell}$ for $\ell<j$. In view of (\ref{E:descol1}), the first term $\chi(z>u)$ of \eqref{eqthm2} contributes the first sum in \eqref{E:gwsigma2}. For $v=N+1-k$, there are $\nu_{i,j}^k-\nu_{i,j}^{k-1}$
        choices of the right-hand entry in row $i$ of $\sigma_j$.
        In each higher rectangle $\sigma_\ell$ for some $\ell>j$, the conditions
        $u\ge v$ and $z>v$ hold in $\nu_{i,\ell}^k$ and
        $\nu_{i+1,\ell}^{k-1}$ columns, respectively.
        Thus the last two terms $\chi(u\ge v)-\chi(z>v)$ of \eqref{eqthm2} contribute
        \[
        (\nu_{i,j}^k-\nu_{i,j}^{k-1})
        (\nu_{i,\ell}^k-\nu_{i+1,\ell}^{k-1}).
        \]
        Summing over $1\le i\le j<\ell\le n$ and $1\le k\le N$ gives the second sum in \eqref{E:gwsigma2}.
 	\end{proof}
	We now prove the four formulas (\ref{eq6})--(\ref{eq64}) of Theorem \ref{T:mono1}.
	
	\subsection{\texorpdfstring{Proof of \eqref{eq6}}{Proof of the first monomial expansion}}
	Combining \eqref{E:quinvc} and Theorem \ref{T:count2} yields
	\begin{align}\label{E:gwc1}
		\bar{\eta}(\sigma)&=\chi_1(\nu)+\sum_{i\le j}\xi_1(s,\nu_{i,j})+\sum_{i\le j}(\lambda_{i}-\lambda_{j})\sum_{k=1}^{N-1}(s_{N}^{k}-s_{k}^{k})\\
		&+\sum_{i\le j}\sum_{k<N}((\nu_{i,j}^{k+1}-s_{k+1}^{1,k}-\nu_{i,j}^{k}+s_{k}^{1,k})(\nu_{i,j}^{k}-s_{k}^{1,k})+\sum_{h\le k}(s_{k+1}^{h}-s_{k}^{h})s_{k}^{h})\notag.
	\end{align}
    Using the identity
    ${s \brack t}_{q^{-1}}=q^{-t(s-t)}{s \brack t}_{q}$, we write every Gaussian factor in $(d^\diamond(\sigma;q^{-1}))$ with parameter $q$. The resulting powers of $q$ cancel precisely the terms in the second line of (\ref{E:gwc1}). Thus (\ref{E:cri1}), (\ref{E:count11}) and (\ref{E:majsigma1}) yield \eqref{eq6}.



	
	\subsection{\texorpdfstring{Proof of \eqref{eq63}}{Proof of the dual monomial expansion}}
	The dual analogue of Lemma \ref{L:sigmaj} uses \eqref{E:defs2} and the nonzero product \eqref{E:coeffdag2}. At each level, the array specifies the non-descent entries below each fixed upper entry. Place these in the rightmost positions of their neutral blocks, in decreasing order; match the remaining descent entries with decreasing upper entries. The Gaussian inequalities ensure the strict inequalities required for these descent columns. The top row is again filled separately in decreasing order. Thus the dual arrays also parameterize dual canonical tableaux uniquely.
	
	Take $\S=\S_5$, so that \eqref{E:count13} applies. Now $s_k^h$ counts the columns $(z,u)$ with $z=h+1$ and $u\le k$. Note that 
	$$\sum_{0\le k<N}(s_{N}^k-s_k^k)$$ 
	counts the non-descent columns $(z,u)$ of $\sigma_j\vert_{i}^{i+1}$ with $\mathrm{leg}(z)=j-i-1$, which leads to
	\begin{align*}
		\sum_{i\le j}\maj(\sigma_j\vert_i^{i+1})=n(\lambda)-\sum_{i\le j}(j-i)\sum_{0\le k<N}(s_{N}^k-s_k^k).
	\end{align*}
    The factor $s_{k+1}^h-s_k^h$ counts the columns $(z,u)$ with $z=h+1\le u=k+1$, while $\nu_{i,j}^k-s_k^{h,k}$ counts the columns $(z,u)$ with $1\le u\le k$ and $1\le z\le h$ or $k+2\le z\le N$. Their product counts the corresponding $\S_5$ quadruples, yielding
   \[
   \eta(\sigma_j\vert_i^{i+1})=\xi_0(s,\nu_{i,j}).
   \]
	The definition of $\S_5$ shows that these quadruples are precisely the ones counted by $\eta(\sigma)$. 
	Importantly, the quadruples $(z,w,u,v)$ and $(w,z,v,u)$ with the pattern $(z>w>v>u)$ of $\S_5$ are excluded, since such quadruples do not occur in any dual canonical tableau $\sigma$.
	
	For triples between unequal-height columns, use \eqref{E:Qindicator}:
	\[
	\Q(z,u,v)=\chi(z\le u)+\chi(u<v)-\chi(z\le v).
	\]
	The first term $\chi(z\le u)$ gives
	\begin{equation}\label{E:dual1}
		\sum_{i\le j}(\lambda_i-\lambda_j)\sum_{0\le k<N}(s_{N}^k-s_k^k).
	\end{equation}
    For a right-hand entry $v=k$ in $\sigma_j$, the other two terms $\chi(u<v)-\chi(z\le v)$ yield
    \[
    (\nu_{i,j}^k-\nu_{i,j}^{k-1})(\nu_{i,\ell}^{k-1}-\nu_{i+1,\ell}^{k})
    \]
    from $\sigma_\ell$, with $\ell>j$. It follows that
    \[
    \eta(\sigma)=\chi_3(\nu)+\sum_{i\le j}
    \bigl(\xi_0(s,\nu_{i,j})+(\lambda_i-\lambda_j)\sum_{0\le k<N}(s_{N}^k-s_k^k)\bigr).
    \]
	Extracting the monomial coefficient from \eqref{E:count13} gives a single sum over the dual tableaux:
	\[
	\CMcal{P}_{\lambda\mu}(q^{-1},t)
	=q^{-n(\lambda')}\sum_{\nu,s}
	t^{\maj(\sigma(\nu,s))}q^{\eta(\sigma(\nu,s))}d_\dagger(\sigma(\nu,s)).
	\]
	Substitute the preceding two statistics and replace $q$ by $q^{-1}$. The definition of $\phi$ then gives \eqref{eq63}.
	
	\subsection{\texorpdfstring{Proof of \eqref{eq62} and \eqref{eq64}}{Proof of the reflected monomial expansions}}
	\begin{definition}[{\em \color{blue} (dual) $\inv$-canonical tableaux}]
		A tableau $\sigma$ is \emph{\color{blue} (dual) $\inv$-canonical} if $\sigma'$ is (dual) canonical, where the map $\sigma\mapsto \sigma'$ is the involution defined in Corollary \ref{cor:10}.
	\end{definition}

Following the notation of Theorem \ref{T:count2}, one has for $\sigma\in\CMcal{C}_{\diamond}(\lambda)$,
\begin{align*}
	\maj(\sigma'_j\big\vert_{j-i}^{j-i+1})&=i\sum_{k=1}^{N-1}(s_N^k-s_k^k),\quad
	\eta^*(\sigma'_j\big\vert_{j-i}^{j-i+1})=\eta(\sigma_j\vert_i^{i+1}),\\
	\overline{\inv}(\sigma')&=\sum_{1\le i\le j\le n}(\lambda_{j-i+1}-\lambda_{j})\sum_{k=1}^{N-1}(s_{N}^{k}-s_{k}^{k}),\notag\\
	&\qquad +\sum_{1\le i\le j\le n}\sum_{k=1}^{N}\sum_{\ell>j}(\nu_{i,j}^{k}-\nu_{i,j}^{k-1})
	(\nu_{\ell-j+i,\ell}^{k}-\nu_{\ell-j+i+1,\ell}^{k-1}),
\end{align*}
where $\overline{\inv}(\sigma')$ denotes the number of non-$\inv$ triples located at two columns of different heights of $\sigma'$. Comparing these statistics with the ones in Theorem \ref{T:count2} and comparing (\ref{E:count12}) with (\ref{E:count11}),  we conclude (\ref{eq62}). In the same manner, (\ref{eq64}) is obtained from (\ref{eq63}) and (\ref{E:count14}).

	

	\section{Final remarks}\label{S:final_remark}
	Canonical tableaux are $\quinv$-sorted tableaux from \cite{O:23,O:24}, but not all $\quinv$-sorted tableaux are canonical. For example,
	\[
	\sigma=\begin{ytableau}5&5&5&5&3&3&3\\4&4&2&2&2&1&4\\2&5&1&3&5&2&1\end{ytableau}
	\]
	is $\quinv$-sorted but not canonical: the neutral block $(2,5,1)$ in the bottom row with common upper entry $4$ violates condition (I) of Definition \ref{Def:3}. In \cite{O:23,O:24}, blocks are contiguous, so $(2,5)$ and $(1)$ are separate blocks. Our blocks may be noncontiguous, imposing conditions on more pairs of columns.
	
	Since $\A\cap\{\inv,\quinv\}=\varnothing$, the compact formulas of Theorem \ref{T:2} differ from \eqref{E:sorted}. It remains interesting to determine whether \eqref{E:sorted} itself yields a positive monomial expansion of the modified Macdonald polynomials.
	
In subsequent work~\cite{JL26}, we use the statistics $\eta\in\A$ introduced here to obtain combinatorial compact formulas for the symmetric Macdonald polynomials $P_\lambda(X;q,t)$ and the nonsymmetric Macdonald polynomials $E_\gamma(X;q,t)$, where $\lambda$ is a partition and $\gamma$ is a weak composition.

	\section*{Acknowledgements}
	The first author was supported by the Austrian Research Fund FWF Elise-Richter Project V 898-N. Both authors were supported by the National Natural Science Foundation of China (NSFC) under Project No. 12201529, and is currently supported by the NSFC under Project No. 12571358.


\end{document}